\documentclass[11pt, reqno]{amsart}
\usepackage{amsthm,amsmath,amsfonts,amssymb,color}
\usepackage[utf8]{inputenc}
\usepackage[bookmarks]{hyperref}
 \usepackage{tikz}
  \usepackage{tikz-cd}
\usetikzlibrary{positioning}
 
 \usepackage{mathtools}
\mathtoolsset{showonlyrefs}

\reversemarginpar

\newtheorem{thm}{Theorem}[section]
\newtheorem{prop}[thm]{Proposition}
\newtheorem{cor}[thm]{Corollary}
\newtheorem{lemma}[thm]{Lemma}

\newtheorem{preremark}[thm]{Remark}
\newenvironment{remark}{\begin{preremark}\rm}{\medskip \end{preremark}}

\numberwithin{equation}{section}

\newcommand{\abs}[1]{\left\vert#1\right\vert}

\newcommand{\R}{\mathbb R}

\newcommand{\al}{\alpha}

\newcommand{\bt}{\beta}

\DeclareMathOperator{\codim}{codim}

\newcommand{\Xm}[1]{\mathfrak X(#1)}

\newcommand{\dd} {\mathrm{d}}

\DeclareMathOperator{\dv}{div}
\DeclareMathOperator{\Def}{Def}
\DeclareMathOperator{\Ric}{Ric}

\DeclareMathOperator{\two}{\mathrm I\mathrm I}
\DeclareMathOperator{\tr}{tr}

\def\be{\begin{equation}}
\def\ee{\end{equation}}
\def\bs{\begin{split}}
\def\ess{\end{split}}

\def\o{\omega}
\newcommand{\vH}{\textbf H}

\newcommand{\ip}[1]{\langle{#1}\rangle}

\newcommand{\tM}{\tilde M}
\newcommand{\tg}{\tilde g}
\newcommand{\tn}{\tilde \nabla}
\def\ag{\abs{\nabla\rho}}

\begin{document}
\title[Gauss formulas for the Laplacians]{The Gauss formulas for Laplacians on submanifolds}
\author[Chan]{Chi Hin Chan}
\address{Department of Applied Mathematics, National Yang Ming Chiao Tung University,1001 Ta Hsueh Road, Hsinchu, Taiwan 30010, ROC}
\email{cchan@math.nctu.edu.tw}

\author[Czubak]{Magdalena Czubak}
\address{Department of Mathematics\\
University of Colorado Boulder\\ Campus Box 395, Boulder, CO, 80309, USA}
\email{czubak@math.colorado.edu}

\begin{abstract}
There are several types of Laplacians of a vector field on a Riemannian manifold.  These include the Bochner and the Hodge Laplacian.  The Gauss formula for the Levi-Civita connection relates the extrinsic connection to the intrinsic connection.  We extend the Gauss formula for the connection to formulas for the different types of Laplacians of a vector field on a submanifold of any codimension $k\geq 1$.  In the process, we derive a Gauss formula for the Ricci operator, formulas for the divergence of the second fundamental form, and a formula for the Laplacian of a $1$-form on a surface of revolution in terms of the Lie derivatives.  The formulas have applications to the study of the formulation of the incompressible Navier-Stokes equations on a Riemannian manifold.
\end{abstract}
%\date{\today}
\maketitle

\tableofcontents
\section{Introduction}

Submanifold geometry studies the interplay between intrinsic and extrinsic structures of manifolds embedded in higher dimensional ambient manifolds. A fundamental identity in this setting is the Gauss formula for the Levi-Civita connection.   The formula decomposes the covariant derivative into a sum of tangential and normal components, and it relates the intrinsic connection to the extrinsic connection.  The goal of this article is to extend this formula to the Gauss formulas for the different Laplacians acting on a vector field on a Riemannian submanifold.  We begin by reviewing the Gauss formula for the Levi-Civita connection.

Let $k\geq 1$, and $(M^n,g)$ be an embedded Riemannian submanifold in $(\tilde M^{n+k},\tg)$, $M^n \hookrightarrow \tilde M^{n+k}$, let $\nabla$ be the Levi-Civita connection on $M$, and let $\tilde \nabla$ be the Levi-Civita connection on $\tilde M$.  Next, let $X, v$ be vector fields on $M$, $p \in M$, and $X, v$ be extensions of $X, v$ to a neighborhood of $p$ in $\tilde M$, still denoted by $X, v$. Then the Gauss formula reads 
 \be\label{GaussCov}
 \tn_X v=\nabla_X v+\two(X, v),
 \ee
where $\two(X, v)$ is the second fundamental form (see Section \ref{prelim} for details), and  where both sides are evaluated at $p\in M$.  The tangential part is given by $\nabla_Xv$ and the normal part is given by the second fundamental form.  The formula shows how the intrinsic covariant derivative of $v$ in the direction of $X$ is related to the extrinsic covariant derivative.  

Similarly, we can relate the curvatures by
\be\label{GaussC}
\tilde R (W,X,Y,Z)=R(W,X,Y,Z)-\tilde g(\two (W,Z),\two(X,Y))+\tilde g(\two(W,Y), \two(X,Z)),
\ee
where again we evaluate on $M$, $\tilde R, R$ are the Riemann curvature tensors on $\tilde M$ and $M$, respectively, and $W, X, Y, Z$ are vector fields on $M$ extended to a neighborhood in $\tilde M$.

In the case of a hypersurface, $M^n\hookrightarrow \tilde M^{n+1}$, these formulas become
 \begin{align}
 \tn_X v&=\nabla_X v+h(X, v)N,\label{Gaussh}\\
 \tilde R(W,X,Y,Z)&=R(W,X,Y,Z)-h (W,Z)h(X,Y)+h(W,Y)h(X,Z),\label{gaussCH}
 \end{align}
where $h(X, v)$ is the scalar second fundamental form (more in Section 2 below).

In this article, we present a comprehensive collection of formulas for the different types of Laplacians of a vector field on a submanifold $M$  in $\tilde M$, as well as formulas for the Ricci operator. These are fundamental operators on manifolds, but to our knowledge such comprehensive study has not been undertaken before.

An important application of these formulas is in the study of incompressible fluids. We explain it now.  More details can be found in \cite{C24}.  

There is no universal agreement what the incompressible Navier-Stokes equations should be on a Riemannian manifold.  One reason for this is that there is more than one choice of the Laplacian that can act on vector fields on a general Riemannian manifold.  For example, there is the Hodge Laplacian and the Bochner (trace) Laplacian.  There is also a third type of operator,
\be\label{divdef}
2\dv \Def
\ee
where $\Def$ is the deformation tensor, which can be written in coordinates as
\be\label{Defd}
 {(\Def u)}_{ij}=\frac 12 (\nabla_i u_j +\nabla_j u_i),
\ee 
where $\nabla$ is the Levi-Civita connection.

The first article to consider the Navier-Stokes equations on Riemannian manifolds was an article of Ebin and Marsden \cite{EbinMarsden}, and the operator \eqref{divdef} was proposed in that article as the ``correct" operator to use when considering the Navier-Stokes equation on an Einstein
manifold.  This operator is indeed used in the Euclidean case, when the equations are derived using the stress tensor (see for example \cite{Batchelor}). In the Euclidean case, this operator reduces to the (standard) Laplacian in the case of the divergence free vector fields, and so do the other Laplacians: Bochner and Hodge.  This can be observed using the Bochner-Weitzenb\"ock formula \cite{Taylor3} and the Ricci identities \cite{Lee_RG} as follows.  Recall, the Bochner Laplacian is given by
\[
-\dv \nabla=\nabla^\ast\nabla,
\]
and the Hodge Laplacian is
 \begin{align}
%-\Delta_B&= \quad\mbox{(Bochner Laplacian)},\\
-\Delta_h&=\dd\dd^\ast+\dd^\ast\dd,\label{HodgeL}
%\quad\mbox{(Hodge Laplacian)},
\end{align}
where $\dd$ is the exterior derivative on differential forms, and $\dd^\ast$ its formal adjoint (also called the divergence).  Then
 \be\label{BW}
 -2\dv\Def u=\nabla^\ast\nabla u+\dd \dd^\ast u-\Ric u=-\Delta_hu+\dd \dd^\ast u-2\Ric u,
 \ee
 where $u$ is a $1$-form, and 
 \[
 \Ric u=\Ric (u^\sharp, \cdot)^\flat,
 \]
 where $\Ric$ is the Ricci tensor, and we use the musical isomorphisms $\sharp, \flat$ to switch between vector fields and $1$-forms \cite{Lee_RG}.

 In \cite{CCM17}, together with Marcelo Disconzi we gave an argument why one would like to use the operator \eqref{divdef}. We referred to the argument as the restriction argument.  The argument can be roughly described as follows: start with a divergence free vector field $\tilde v$ on $\R^3$ that is tangential to the sphere, apply the Euclidean Laplacian to the vector field, restrict  the Laplacian to the sphere, and see what the resulting operator is.  The computation in \cite{CCM17} showed that it was the operator \eqref{divdef}.

Then further work in an article \cite{CCY_22} with Tsuyoshi Yoneda revealed that the situation is more complicated than initially expected.  In \cite{CCY_22} we worked with an ellipsoid.  And while the ellipsoid is no longer an Einstein manifold, which created technical difficulties, a more important point was revealed which is relevant even in the case of the sphere: the operator obtained depends on the  vector field $\tilde v$ on $\R^3$.  There is more than one vector field that can restrict to the same tangent vector field on the submanifold $M$ in question and depending on the vector field $\tilde v$ a different operator can be obtained.  
 
For example, in \cite{CCM17} we worked with a vector field that at the time, appeared to be the most natural way to write a vector field that restricts to a vector field on the sphere. Also, that vector field can be viewed as an extension of the vector field on the sphere to $\R^3$ so its norm grows depending on the distance away from the sphere.  If we work with a vector field that has a norm preserved, then one could obtain the Hodge Laplacian instead!  This was actually already pointed out in \cite{Orszag, Yamada_2018}.  What does this mean?  At this point, we do not know, but at the very least, it means that the argument from \cite{CCM17} is not conclusive, and the problem merits further investigations.
  
The article \cite{CCM17} and \cite{Orszag, Yamada_2018} were all done for the case of the sphere, and considered specific extensions: one type of extension in \cite{CCM17} and two in \cite{Orszag, Yamada_2018}.  The paper \cite{CCY_22} extended the prior work to an ellipsoid, and it produced a formula that worked for any general extension vector field.  In addition, it generalized the work of \cite{CCM17, Orszag, Yamada_2018} as it could also cover the case of the sphere.  

The observation that we make in this article is that the restriction argument described above can be viewed as a special case of the Gauss formula, where we focus on the tangential component. We now keep track of both of the tangential and normal components and consider different type of Laplacians and general ambient manifolds, not just a Euclidean space, as well as any codimension $k\geq 1$. As a result, we obtain the Gauss formulas for Laplacians on submanifolds.  Moreover, we do not need a divergence free condition which was required in \cite{CCY_22} although we write down some additional formulas in the case of the divergence free vector fields, too.  Finally, using the formula for a hypersurface, we are able to extend the Lie derivative formula from \cite{CCY_22} to a surface of revolution.

\subsection{Main Results}
The first main result is the Gauss formula for the Bochner Laplacian, and the second main result is the Gauss formula for the Ricci operator.  Using these two formulas we can obtain as corollaries formulas for the deformation Laplacian and the Hodge Laplacian.  From here, we can obtain the formulas for the projected Laplacians, which was the original motivation for this work.

When the submanifold is a Euclidean submanifold, meaning it is embedded in a Euclidean space, then the formulas are often simpler.  For that reason, and because it might be useful for applications, we record all the formulas separately for the Euclidean submanifolds.

In the case of a hypersurface, when the codimension is one, the formulas can be also simplified from the general case of the higher codimension.  Therefore that case is explicitly stated, too.  In this case, there can be additional simplifications for vector fields that are divergence free, so we highlight these as well.

The third main result is the extension of the formula obtained in \cite{CCY_22} for the ellipsoid  in terms of the Lie derivatives of a $1$-form.  We extend that formula to a surface of revolution in $\R^3$.  The formula follows from the main theorem for the Bochner Laplacian, but its derivation is more complex, which is why it is a theorem, and not simply a corollary.

The presentation is motivated by our point of view that we would like to provide a convenient collection of formulas and to minimize the effort of the reader, and present the formulas as explicitly as possible.  To that end we include equivalent versions of the formulas.  Another benefit of including equivalent versions of the formulas is that different versions can guide us in choosing the viscosity operator we would like to work with; see Section \ref{closerlook} for a detailed discussion.  Finally, different forms of the formulas can be used in applications to boundary value problems.

\subsubsection{Bochner Laplacian formulas}

We begin by stating the Gauss formula for the Bochner Laplacian for a general embedded submanifold $M\hookrightarrow \tilde M$.  It can be stated in two different but equivalent ways.
Then, we specialize to a hypersurface.

\begin{thm}[Gauss formula for the Bochner Laplacian]\label{gaussThmM}
Let $n\geq 2$, $k\geq 1$, $(\tilde M, \tilde g)$ be a Riemannian manifold, $M$ an embedded submanifold, $M\hookrightarrow \tilde M$, $\dim M=n$, $\dim \tilde M=n+k$,
and let $v$ be a vector field on $M$, $p \in M$, and $v$ be an extension of $v$ to a neighborhood of $p$ in $\tilde M$, still denoted by $v$.  Then, evaluating at $p \in M$ we have

\be\label{GaussB1}
\begin{split}
\tilde \nabla^\ast \tilde \nabla v &=\nabla^\ast\nabla v  +\sum_{l=1}^k W_{N_l}W_{N_l}v-nW_{\vH}v\\
&\quad+n[\vH, v]-\sum_{l=1}^k\tn_{N_l}\tn_{N_l}v+\sum_{l=1}^k\tn_{\tn_{N_l}N_l}v,\\
&\qquad -2\delta^{ij}\two(E_i, \nabla_{E_j}v)-\delta^{ij}(\tilde R(E_j,v)E_i)^\perp,
\end{split}
\ee
or equivalently
 
\be\label{GaussB2}
\begin{split}
\tn^\ast \tn v &=\nabla^\ast\nabla v+\sum_{l=1}^k W_{N_l}W_{N_l}v\\
&\quad+n \tn_{\vH}v-\sum_{l=1}^k\tn_{N_l}\tn_{N_l}v+\sum_{l=1}^k\tn_{\tn_{N_l}N_l}v\\
&\qquad -(\tr \nabla^B \two) (v)-2\delta^{ij}\two(E_i, \nabla_{E_j}v).
\end{split}
\ee
 In the case of a hypersurface, $M^n\hookrightarrow \tilde M^{n+1}$, \eqref{GaussB1} reduces to

\be\label{b4projectB}
\begin{split}
\tilde \nabla^\ast \tilde \nabla v &=\nabla^\ast  \nabla v+s^2v+nH\tn_Nv-\tilde \nabla_N\tilde\nabla_N v+\nabla_{\tilde\nabla_N N}v\\
&\qquad -\left(2\sum_i\kappa_i \nabla_i v^i  +nv(H)-h(\tn_NN, v)\right)N+(\widetilde\Ric v)^\perp,
\end{split}
\ee
and \eqref{GaussB2} reduces to
\be\label{b4projectB2}
\begin{split}
\tilde \nabla^\ast \tilde \nabla v &=\nabla^\ast  \nabla v+s^2v+nH\tn_Nv-\tilde \nabla_N\tilde\nabla_N v+\nabla_{\tilde\nabla_N N}v\\
&\qquad -\big((\dv h)(v)+2\sum_i\kappa_i \nabla_i v^i-h(\tn_NN, v)\big) N.
\end{split}
\ee
Here
\begin{itemize}
\item $\tn$ denotes the  Levi-Civita  connection on $\tilde M$,
\item $\nabla$ denotes the Levi-Civita connection on $M$,
\item $\{N_l\}, \ l=1, \dots, k$ is the orthonormal basis for the normal bundle, defined in a neighborhood of $p\in M$,
\item for $k=1,$ $N$ is the choice of the unit normal to $M$, defined in a neighborhood of $p\in M$,
\item for $k\geq 2$, $E_1, \dots, E_n$ is any ON frame for $M$, defined in a neighborhood of $p\in M$, so that $E_1, \dots, E_n, N_1, \dots, N_k$ is an ON frame for $\tilde M$,
\item for $k=1$, $E_1, \dots, E_n$ is a local frame so that $E_i$ are eigenvectors at $p\in M$ associated to the shape operator (see Remark \ref{rem6})
\item $W_{N}$ is the Weingarten map associated to a normal vector field $N$  (see Section \ref{so_setup}),
\item $s$ is the shape operator (see Section \ref{so_setup}),
\item $H$ is the mean curvature of $M$, and $\vH$ is the mean curvature vector  (see Section \ref{Subk}), 
\item $\nabla^B$ is the connection for bilinear maps  (see Section \ref{Subk}), and  $( \tr \nabla^B \two)(v)=(\delta^{ij}\nabla^B_{E_i}\two)(E_j,v)$,
\item $\kappa_i$ are the principal curvatures of $M^n\hookrightarrow \tilde M^{n+1}$,
\item the notation $\nabla_iv^i$ is defined in \eqref{best_notation},
\item $\perp$ denotes the orthogonal projection on the normal space $(T_pM)^\perp$,
\item $\tilde R$ denotes the Riemann curvature on $\tilde M$,
\item $\widetilde \Ric v=(\widetilde \Ric (v,\cdot))^\sharp$, $\widetilde \Ric(\cdot, \cdot)$ is the Ricci tensor on $\tilde M$.
 \end{itemize}

\end{thm}
\begin{remark}
The above formulas do not require $v$ to be divergence free, but see Corollary \ref{corDefdiv}.
\end{remark}
\begin{remark}\label{Rem3}
If $v$ is divergence free both on $M^n$ and $\tilde M^{n+1}$, then the term $nH\tn_Nv$ in \eqref{b4projectB} and \eqref{b4projectB2} is automatically tangential.
\end{remark}

\begin{remark}
We note the above formulas do not explicitly contain the Ricci operator that we would like to see if we are interested in connecting the projected operator with the deformation Laplacian \eqref{divdef}, but see Corollary \ref{cBE}.
\end{remark}
\begin{remark}
The formulas \eqref{GaussB2} and \eqref{b4projectB2} remain the same regardless if the ambient manifold is an abstract manifold or the Euclidean space (see Corollary \ref{cBE}).
\end{remark}
\begin{remark}\label{rem6}
The formulas are indpendent of the choice of the (adapted) ON frame $\{E_i, N_l \}$.  Moreover, this can be said about each individual term that appears in the formulas with the following exceptions.  The expression $\sum_{l=1}^k\tn_{N_l}\tn_{N_l}v-\sum_{l=1}^k\tn_{\tn_{N_l}N_l}v$ needs to be considered as one term in order to be independent of the frame.  In fact, sometimes the terms in the sum are thought of as one operator $$\tn^2_{N_l, N_l}:=\tn\tn v(\cdot, N_l, N_l)=\tn_{N_l}\tn_{N_l}v-\tn_{\tn_{N_l}N_l}v,$$  (see e.g., \cite[p.99]{Lee_RG}).  

The simple presentation of the term $\sum_i \kappa_i \nabla_i v^i$, which can be thought of as weighted divergence, relies on the fact that the frame consists of the eigenvectors of the shape operator.  In general, it can be written as $\delta^{ij}h(E_i, \nabla_{E_i}v)$ (see \eqref{generalF1}) which is easily seen to be independent of the frame. Futhermore, the term $\delta^{ij}h(E_i, \nabla_{E_i}v)N$ is a special case of $\delta^{ij}\two(E_i, \nabla_{E_i}v)$, and it is interesting to observe that we could write it in a generalization of $\sum_i \kappa_i \nabla_i v^iN$ to any frame and any codimension $k\geq 1$ as follows
\[
\sum_{i=1}^n\two(E_i, \nabla_{E_i}v)=-\sum_{l=1}^k\sum_{i,j=1}^n(\Def v)_{ij}c^j_{i(n+l)}N_l,
\]
where $c^\alpha_{\beta \gamma}$ are the coefficients of the Lie brackets of the frame (see \eqref{cbracket}), and $\Def v$ is given by \eqref{Defd}.

Finally, in general, checking independence of the frame can be done by following similar steps as in Lemma \ref{indep} and we leave out the details for other terms.
 \end{remark}

From Theorem \ref{gaussThmM} we have a simple corollary giving a formula for the second fundamental form that we have not seen before.
\begin{cor}\label{corsimple} Let $M^n\hookrightarrow \tilde M^{n+1}$ be an embedded hypersurface and $v\in \mathfrak X(M)$.  Then the scalar second fundamental form $h$ satisfies
\be\label{divh}
(\dv h)(v)=nv(H)-\widetilde \Ric (v,N),
\ee
where $H$ is the mean curvature of $M$.  In general, for $M^n\hookrightarrow \tilde M^{n+k}, k\geq 1$, we have
\be\label{divII}
(\tr \nabla^B \two)(v)=n\nabla^\perp_v\vH+\delta^{ij}(\tilde R(E_i,v)E_j)^\perp,
\ee
where $\{E_i\}$ is any ON frame on $M$.
\end{cor}
\begin{remark}
One consequence of \eqref{divh}, for example, is that divergence of $h$ for Euclidean hypersurfaces is simply $nv(H)$, which further reduces to zero for surfaces with constant mean curvature. 
\end{remark}

\subsubsection{Ricci formula and its application to the Euclidean submanifolds}
Next, we present the Gauss formula for the Ricci operator.  The formula is of its own interest.  In addition, it has multiple applications in this article.  It can be used to extend the above Gauss formulas to the deformation Laplacian and the Hodge Laplacian, as well as to write equivalent formulas for the Bochner Laplacian in the case of the Euclidean submanifolds.
 
\begin{thm}[Gauss Formula for the Ricci Operator]\label{ThmR}
Let $(M^n, g)\hookrightarrow (\tilde M^{n+k}, \tilde g)$, $v$ be a vector field on $M$, $p \in M$, and $v$ be an extension of $v$ to a neighborhood of $p$ in $\tilde M$, still denoted by $v$.  Let $\widetilde \Ric$, $\Ric$ be the Ricci operators obtained by raising an index in the corresponding Ricci tensor, i.e., 
$$\widetilde \Ric v=\widetilde \Ric(v,\cdot)^\sharp,\quad\Ric v= \Ric(v,\cdot)^\sharp.$$ 
Then evaluating at $p\in M$, we have
\be\label{genRicM}
\widetilde \Ric v=\Ric v-nW_\vH v+\sum_{l=1}^k\left(W_{N_l}W_{N_l}v+\delta^{ij}\tilde R(N_l, E_i, v, N_l)E_j\right)+(\widetilde \Ric v)^\perp.
\ee
Hence for $M \hookrightarrow \R^{n+k}$
\be\label{genRicR}
\Ric v=nW_\vH v-\sum_{l=1}^kW_{N_l}W_{N_l}v.
\ee

Then, for an embedded hypersurface, $M^n\hookrightarrow \tilde M^{n+1}$, evaluating on $M$, it holds
\be\label{genRic}
\widetilde \Ric v=\Ric v-nHsv+s^2v+\delta^{ij}\tilde R(N, E_i, v, N)E_j+(\widetilde \Ric v)^\perp,
\ee
and so for a hypersurface $M\hookrightarrow \R^{n+1}$
\be\label{Ricci1}
\Ric v=nHsv-s^2v.
\ee

\end{thm}
\begin{remark}
Formula \eqref{Ricci1} can be found, for example, in \cite[p.35]{KB2}, but to our knowledge, formulas \eqref{genRicM}-\eqref{genRic} are new.
\end{remark}
\begin{remark}
Here, we can see a clear breakdown into the tangential and normal parts: the first four terms in both \eqref{genRicM} and \eqref{genRic} are all tangential while the last term is orthogonal to $M$.  Moreover, by the tensorial properties of $\tilde R$ and $\widetilde \Ric$, all the parts of the formula are independent of the extension of $v$, in fact, only depend on $v|_p, p \in M$.
\end{remark}

In the case of Euclidean submanifolds, we can write the formulas in Theorem \ref{gaussThmM} as follows.
 \begin{cor}[Bochner Laplacian for Euclidean submanifolds]\label{cBE}
 Let $\Ric v=( \Ric (v,\cdot))^\sharp$,  where $\Ric(\cdot, \cdot)$ is the Ricci tensor on $M$.
 With the notation as in Theorem \ref{gaussThmM}, let $M^n\hookrightarrow \R^{n+k}$.  Then, evaluating at $p \in M$ the formula \eqref{GaussB1} becomes
\be\label{GaussB1R}
\begin{split}
\tilde \nabla^\ast \tilde \nabla v &=\nabla^\ast\nabla v  +\sum_{l=1}^k W_{N_l}W_{N_l}v-nW_{\vH}v+n[\vH, v]\\
&\qquad -\sum_{l=1}^k\tn_{N_l}\tn_{N_l}v+\sum_{l=1}^k\tn_{\tn_{N_l}N_l}v-2\delta^{ij}\two(E_i, \nabla_{E_j}v),
\end{split}
\ee
and can be also written as
\be\label{GaussB1R2}
\begin{split}
\tilde \nabla^\ast \tilde \nabla v &=\nabla^\ast\nabla v  -\Ric v+n[\vH, v]-\sum_{l=1}^k\tn_{N_l}\tn_{N_l}v+\sum_{l=1}^k\tn_{\tn_{N_l}N_l}v\\
&\qquad-2\delta^{ij}\two(E_i, \nabla_{E_j}v).
\end{split}
\ee
Formula \eqref{GaussB2} stays the same or it can be written as

 \be\label{GaussB2R}
\begin{split}
\tn^\ast \tn v &=\nabla^\ast\nabla v-\Ric v+nW_{\vH}v+n \tn_{\vH}v-\sum_{l=1}^k\tn_{N_l}\tn_{N_l}v+\sum_{l=1}^k\tn_{\tn_{N_l}N_l}v\\
&\qquad -(\tr \nabla^B \two) (v)-2\delta^{ij}\two(E_i, \nabla_{E_j}v).
\end{split}
\ee
In the case of a hypersurface,  $M\hookrightarrow \R^{n+1}$,  we have directly from \eqref{b4projectB}  
 
\be\label{b4projectRBe}
\begin{split}
\tilde \nabla^\ast \tilde \nabla v &=\nabla^\ast  \nabla v+s^2v+nH\tn_Nv-\tilde \nabla_N\tilde\nabla_N v+\nabla_{\tilde\nabla_N N}v\\
&\qquad -\left(2\sum_i\kappa_i \nabla_i v^i  +nv(H)-h(v,\tn_NN)\right)N,
\end{split}
\ee
and written using $\Ric v$, giving a version of \eqref{GaussB2R},
\be\label{b4projectRBe2}
\begin{split}
\tilde \nabla^\ast \tilde \nabla v &=\nabla^\ast  \nabla v-\Ric v+nH[N,v]-\tilde \nabla_N\tilde\nabla_N v+\nabla_{\tilde\nabla_N N}v\\
&\qquad -\left(2\sum_i\kappa_i \nabla_i v^i  +nv(H)-h(v,\tn_NN)\right)N,
\end{split}
\ee
while \eqref{b4projectB2} remains the same, or with the $\Ric v$ term reading as
\be\label{b4projectRB2}
\begin{split}
\tilde \nabla^\ast \tilde \nabla v &=\nabla^\ast  \nabla v-\Ric v+nH[N,v]-\tilde \nabla_N\tilde\nabla_N v+\nabla_{\tilde\nabla_N N}v\\
&\qquad -\left((\dv h)(v)+2\sum_i\kappa_i \nabla_i v^i -h(v,\tn_NN)\right) N.
\end{split}
\ee
\end{cor}
\begin{remark}
We note that while the formulas \eqref{GaussB1R} and \eqref{b4projectRBe} are obtained by dropping the ambient curvature terms in \eqref{GaussB1} and  \eqref{b4projectB}, respectively, the formulas \eqref{GaussB1R2}, \eqref{GaussB2R}, \eqref{b4projectRBe2} and \eqref{b4projectRB2} rely on the connection of the Ricci operator with the shape operator as written in Theorem \ref{ThmR} (formulas \eqref{genRicR} and \eqref{Ricci1}).  This connection allows for another formulation of the formulas which accentuate different terms.  

Of course, using \eqref{genRicM} and \eqref{genRic}, one could also write more versions of the Gauss formulas for a general ambient manifold $\tilde M$, but this would introduce a term that is neither purely tangential nor normal, the term $\widetilde \Ric v$, so for that reason, we choose not to explicitly write these formulas.
\end{remark}

\begin{remark}
 Similarly as in Remark \ref{Rem3}, if $v$ is divergence free both on $M^n$ and $\R^{n+1}$, then the term $nH\tn_Nv$ in \eqref{b4projectRBe} is tangential.  In addition, in that case, the same holds for $[N,v]$ in \eqref{b4projectRBe2} and \eqref{b4projectRB2}.  
\end{remark}

\begin{remark}
For a hypersurface, we could also replace $nH[N,v]$ with $nH\tn_Nv+sv$.   
\end{remark}

\subsubsection{Deformation Laplacian formulas}

Using the formulas for the Bochner Laplacian and the Ricci operator we can obtain formulas for the deformation Laplacian. 
 
\begin{cor}[Gauss formula for the deformation Laplacian]\label{corDef} With the notation as in Theorem \ref{gaussThmM} let
\begin{align}
\tilde Lv&=-2\widetilde\dv\widetilde\Def v,\quad Lv=-2\dv\Def v.
 \end{align}
Then evaluating on $M^n\hookrightarrow \tilde M^{n+k}$, we have
\be\label{GaussDef}
\begin{split}
\tilde L v  &=L v+n[\vH, v]-\sum_{l=1}^k\tn_{N_l}\tn_{N_l}v+\sum_{l=1}^k\tn_{\tn_{N_l}N_l}v+\mathcal E_1(v)+\mathcal N_1(v),
\end{split}
\ee
or equivalently
\be\label{GaussDef2}
\begin{split}
\tilde L v  &=L v+nW_\vH v+n \tn_{\vH}v-\sum_{l=1}^k\tn_{N_l}\tn_{N_l}v+\sum_{l=1}^k\tn_{\tn_{N_l}N_l}v+\mathcal E_1(v)+\mathcal N_2(v),
\end{split}
\ee
where
\begin{align*}
\mathcal E_1(v)&=- \sum_{l=1}^k\delta^{ij}\tilde R(N_l, E_i, v, N_l)E_j+\nabla \dv^\perp v,\\
\mathcal N_1(v)&=-2\delta^{ij}\two(E_i, \nabla_{E_j}v)-\delta^{ij}(\tilde R(E_j,v)E_i)^\perp-(\widetilde \Ric v)^\perp+\nabla^\perp \widetilde \dv v,\\
\mathcal N_2(v)&=-(\tr \nabla^B \two) (v)-2\delta^{ij}\two(E_i, \nabla_{E_j}v)-(\widetilde \Ric v)^\perp+\nabla^\perp \widetilde \dv v,
\end{align*}
where $\dv^\perp$ and $\nabla^\perp$ are defined in Section \ref{orthdg}.

 In the case of a hypersurface, $M^n\hookrightarrow \tilde M^{n+1}$,
  
\be\label{b4projectBdef}
\begin{split}
\tilde L v&=Lv+nH[N,v]-(\tilde \nabla_N\tilde\nabla_N v)^T+\nabla_{\tilde\nabla_N N}v+\mathcal E_2(v)+\mathcal N_3(v),
\end{split}
\ee
where
\begin{align*}
\mathcal E_2(v)&=-\delta^{ij}\tilde R(N, E_i, v, N)E_j+\nabla\dv^\perp v,\\
\mathcal N_3(v)&=-\Big(2\sum_i\kappa_i \nabla_i v^i  +nv(H)\Big)N+\Big(N (\tn_i v^i)+\tg([N,v], \tn_NN)\Big)N,\\
	&=-\Big(2\sum_i\kappa_i \nabla_i v^i +\dv h(v)\Big) N-(\widetilde \Ric v)^\perp+\Big(N (\tn_i v^i)+\tg([N,v], \tn_NN)\Big)N.
\end{align*}
\end{cor}
\begin{remark}
In the case of a hypersurface, there is a cancellation with the normal component of $\tn_N\tn_Nv$, which is why we are left only with the tangential part.
\end{remark}
It is easy to see what happens with the above formulas for the deformation Laplacian when the ambient manifold is the Euclidean space: they stay exactly the same but we lose any terms involving $\tilde R$ and $\widetilde \Ric$.  For simplicity, we omit writing these explicitly.

It is interesting to see how the formulas simplify in the case when the vector fields are divergence free: both on $\tilde M$ and on $M$.  The formulas are the most interesting in the case of codimension $1$.  For higher codimensions, in case of the ambient manifold being $\R^{n+k}$, due to the Bochner Laplacian and the deformation Laplacian coinciding for divergence free vector fields, the formulas \eqref{GaussB1R2} and \eqref{GaussB2R} can be also derived from the deformation Laplacian formulas \eqref{GaussDef} and \eqref{GaussDef2}, respectively.  We record this below.

\begin{cor}[Gauss formula for the deformation Laplacian: divergence free case]\label{corDefdiv} With the notation as above, suppose $v$ is divergence free both on $M$ and on $\tilde M$.  Then for a general submanifold we have that the form of \eqref{GaussDef} and \eqref{GaussDef2} stays the same with the deformation Laplacian simplifying to just Bochner Laplacian minus Ricci term, and all the last terms in $\mathcal E_1, \mathcal N_1, \mathcal N_2$ are zero.  In particular, for $M\hookrightarrow \R^{n+k}$ we have two equivalent formulas \eqref{GaussB1R2} and \eqref{GaussB2R}.
 
In the case of a hypersurface, $M^n\hookrightarrow \tilde M^{n+1}$, we have
 \be\label{b4projectBdefdiv}
\begin{split}
\tn^\ast\tn v-\widetilde\Ric v&=\nabla^\ast\nabla v-\Ric v+nH[N,v]^T-\tilde \nabla_N\tilde\nabla_N v+\nabla_{\tilde\nabla_N N}v+\mathcal E_2(v)+\mathcal N_4(v),
\end{split}
\ee
where
\begin{align*}
\mathcal E_2(v)&=-\delta^{ij}\tilde R(N, E_i, v, N)E_j,\\
\mathcal N_4(v)&=-\Big(2\sum_i\kappa_i \nabla_i v^i+nv(H) \Big)N+h(\tn_NN, v)N\\
&=-\big(2\sum_i\kappa_i \nabla_i v^i+(\dv h)(v)\big) N-(\widetilde \Ric v)^\perp+h(\tn_NN, v)N.
\end{align*}
Hence for a hypersurface $M^n\hookrightarrow \R^{n+1}$
\be\label{b4projectBdefdivR}
\begin{split}
\tn^\ast\tn v&=\nabla^\ast\nabla v-\Ric v+nH[N,v]^T-\tilde \nabla_N\tilde\nabla_N v+\nabla_{\tilde\nabla_N N}v+\mathcal N_4(v).
\end{split}
\ee
 \end{cor}
 \begin{remark}
 In the case of a Euclidean hypersurface, we can compare the formula \eqref{b4projectBdefdivR} to \eqref{b4projectRBe2} and \eqref{b4projectRB2}. We see the formulas are the same except that the divergence free condition allows us to show the term $nH[N,v]$ is tangential.
 \end{remark}
\subsubsection{Hodge Laplacian formulas}
We also include formulas for the Hodge Laplacian.  
\begin{cor}[Gauss formula for the Hodge Laplacian]\label{corHodgeLap} Let $-\tilde \Delta_h$ denote the Hodge Laplacian on $\tilde M$, and similarly let $-\Delta_h$ denote the Hodge Laplacian on $M$.  Then

\be
\bs
-\tilde \Delta_hv&=- \Delta_h v +2\sum_{l=1}^k W_{N_l}W_{N_l}v-2nW_{\vH}v+\sum_{l=1}^k\delta^{ij}\tilde g(\tilde R(N_l, E_i) v, N_l)E_j\\
&\quad +n[\vH, v]-\sum_{l=1}^k\tn_{N_l}\tn_{N_l}v+\sum_{l=1}^k\tn_{\tn_{N_l}N_l}v\\
&\quad -2\delta^{ij}\two(E_i, \nabla_{E_j}v)-\delta^{ij}(\tilde R(E_j,v)E_i)^\perp+(\widetilde \Ric v)^\perp,
\end{split}
\ee
or equivalently
\be
\bs
-\tilde \Delta_hv&=- \Delta_h v +2\sum_{l=1}^k W_{N_l}W_{N_l}v-nW_{\vH}v+\sum_{l=1}^k\delta^{ij}\tilde g(\tilde R(N_l, E_i) v, N_l)E_j\\
&\quad +n\tn_{\vH}v-\sum_{l=1}^k\tn_{N_l}\tn_{N_l}v+\sum_{l=1}^k\tn_{\tn_{N_l}N_l}v\\
&\quad -(\tr \nabla^B \two) (v)-2\delta^{ij}\two(E_i, \nabla_{E_j}v)+(\widetilde \Ric v)^\perp.
\end{split}
\ee
 In the case of a hypersurface, $M^n\hookrightarrow \tilde M^{n+1}$, we have  
 \be\label{HodgeLapH}
\begin{split}
-\tilde \Delta_hv &= -\Delta_hv+2s^2v+nH(\tn_Nv-sv)+\delta^{ij}\tilde R(N, E_i, v, N)E_j\\
&\qquad -\tilde \nabla_N\tilde\nabla_N v+\nabla_{\tilde\nabla_N N}v+\mathcal N_5(v),
\end{split}
\ee
where
\begin{align*}
\mathcal N_5(v)&=  -\left(2\sum_i\kappa_i \nabla_i v^i  +nv(H)-h(\tn_NN, v)\right)N+2(\widetilde\Ric v)^\perp\\
&=-\big((\dv h)(v)+2\sum_i\kappa_i \nabla_i v^i-h(\tn_NN, v)\big) N+(\widetilde\Ric v)^\perp.
\end{align*}
\end{cor}
 
 \subsubsection{Projected Laplacian formulas}
The original motivation for this work was to explore the formula for the projected Laplacian.  We do this now.

\begin{cor}[Projected Laplacians] \label{corProject}
With the notation as above, for a general submanifold $M^n\hookrightarrow \tilde M^{n+k}$, for the Bochner Laplacian we have
\be\label{GaussB1p}
(\tilde \nabla^\ast \tilde \nabla v)^T =\nabla^\ast\nabla v  +\sum_{l=1}^k W_{N_l}W_{N_l}v+n(\tn_{\vH}v)^T-\sum_{l=1}^k(\tn_{N_l}\tn_{N_l}v- \tn_{\tn_{N_l}N_l}v)^T.
\ee
 In the case of a hypersurface, $M^n\hookrightarrow \tilde M^{n+1}$, we have

\be\label{projectB}
(\tilde \nabla^\ast \tilde \nabla v)^T =\nabla^\ast  \nabla v+s^2v+nH(\tn_Nv)^T-(\tilde \nabla_N\tilde\nabla_N v)^T+\nabla_{\tilde\nabla_N N}v.
\ee

For the deformation Laplacian, and $M^n\hookrightarrow \tilde M^{n+k}$, we have

\be\label{DefprojectM}
\begin{split}
(\tilde L v)^T  &=L v+n[\vH, v]^T-\sum_{l=1}^k(\tn_{N_l}\tn_{N_l}v- \tn_{\tn_{N_l}N_l}v)^T\\
&\qquad- \sum_{l=1}^k\delta^{ij}\tilde R(N_l, E_i, v, N_l)E_j+\nabla \dv^\perp v,
\end{split}
\ee
and for a hypersurface,  $M^n\hookrightarrow \tilde M^{n+1}$,

\be\label{projectBdef2}
(\tilde Lv)^T=Lv+nH[N,v]^T-\tilde \nabla_N\tilde\nabla_N v+\nabla_{\tilde\nabla_N N}v -\delta^{ij}\tilde R(N, E_i, v, N)E_j+\nabla\dv^\perp v.
\ee

For divergence free vector field, with both $\widetilde \dv v=0=\dv v$, for $M^n\hookrightarrow \tilde M^{n+k}$, we have

\be\label{Defprojectdiv}
\begin{split}
(\tilde \nabla^\ast \tilde \nabla v-\widetilde \Ric v)^T  &=\nabla^\ast  \nabla v-\Ric v+n[\vH, v]^T-\sum_{l=1}^k(\tn_{N_l}\tn_{N_l}v- \tn_{\tn_{N_l}N_l}v)^T\\
&\qquad- \sum_{l=1}^k\delta^{ij}\tilde R(N_l, E_i, v, N_l)E_j,
\end{split}
\ee
and for a hypersurface,  $M^n\hookrightarrow \tilde M^{n+1}$, 
  \be\label{projectBdefdivfree}
\begin{split}
(\tn^\ast\tn v-\widetilde\Ric v)^T&=\nabla^\ast\nabla v-\Ric v+nH[N,v]-(\tilde \nabla_N\tilde\nabla_N v)^T+\nabla_{\tilde\nabla_N N}v\\
&\quad-\delta^{ij}\tilde R(N, E_i, v, N)E_j.
\end{split}
\ee 

For the Hodge Laplacian, and  $M^n\hookrightarrow \tilde M^{n+k}$, we have
 
\be
\begin{split}
(-\tilde \Delta_hv )^T&=- \Delta_h v +2\sum_{l=1}^k W_{N_l}W_{N_l}v+n (\tn_{\vH}v)^T-nW_{\vH}v\\
&\qquad -\sum_{l=1}^k(\tn_{N_l}\tn_{N_l}v- \tn_{\tn_{N_l}N_l}v)^T
+ \sum_{l=1}^k\delta^{ij}\tilde R(N_l, E_i, v, N_l)E_j,
\end{split}
\ee
and for a hypersurface,  $M^n\hookrightarrow \tilde M^{n+1}$, 
  \be\label{projectH1}
\begin{split}
(-\tilde \Delta_hv )^T&= -\Delta_hv+2s^2v+nH(\tn_{N}v)^T-nHsv-(\tilde \nabla_N\tilde\nabla_N v)^T+\nabla_{\tilde\nabla_N N}v\\
&\quad+\delta^{ij}\tilde R(N, E_i, v, N)E_j.
\end{split}
\ee 
 \end{cor}

\begin{remark}
In formula \eqref{projectBdef2}, the term $\tn_N\tn_N v$ is already tangential, before the projection, and we leave it with no projection here since it is not necessary (cf. \eqref{b4projectBdef}).  Similarly, in \eqref{projectBdefdivfree}, the term $[N,v]$ is also tangential (cf. \eqref{b4projectBdefdiv}).
\end{remark}
\begin{remark}\label{remb}
By the Weingarten equation \eqref{Weq2}, in \eqref{Defproject} we could equivalently write $n[\vH, v]^T=nW_{\vH}v+n(\tn_{\vH}v)^T.$
\end{remark}
We specialize now to the Euclidean submanifolds.  Since the different Laplacians coincide, we organize the presentation of the formulas according to the codimension instead of the type of the Laplacian.
\begin{cor}[Projected Laplacians: Euclidean submanifolds]\label{corProjectE}
For Euclidean submanifolds, we have the following formulas.
\begin{itemize}
\item For $M^n\hookrightarrow \R^{n+k}$, 
\begin{itemize}
\item for the Bochner Laplacian, the formula \eqref{GaussB1p} stays the same, it reads

\be\label{GaussB1pagain}
(\tilde \nabla^\ast \tilde \nabla v)^T =\nabla^\ast\nabla v  +\sum_{l=1}^k W_{N_l}W_{N_l}v+n(\tn_{\vH}v)^T-\sum_{l=1}^k(\tn_{N_l}\tn_{N_l}v- \tn_{\tn_{N_l}N_l}v)^T,
\ee
and it can be equivalently written as
\be\label{projectBR1}
\begin{split}
(\tilde \nabla^\ast \tilde \nabla v)^T &=\nabla^\ast\nabla v  -\Ric v +n[\vH, v]^T-\sum_{l=1}^k(\tn_{N_l}\tn_{N_l}v- \tn_{\tn_{N_l}N_l}v)^T.
\end{split}
\ee

From the Hodge Laplacian we get another equivalent formula
\be\label{Hodgepk}
\begin{split}
(\tilde \nabla^\ast \tilde \nabla v)^T&=- \Delta_h v +2\sum_{l=1}^k W_{N_l}W_{N_l}v+n (\tn_{\vH}v)^T-nW_{\vH}v \\
&\quad-\sum_{l=1}^k(\tn_{N_l}\tn_{N_l}v- \tn_{\tn_{N_l}N_l}v)^T.
\end{split}
\ee
\item  For the deformation Laplacian, we have

\be\label{Defproject}
\begin{split}
(\tilde \nabla^\ast \tilde \nabla v+\tilde \nabla \widetilde \dv v)^T  &=\nabla^\ast  \nabla v-\Ric v+  \nabla   \dv v +n[\vH, v]^T\\
&\quad -\sum_{l=1}^k(\tn_{N_l}\tn_{N_l}v- \tn_{\tn_{N_l}N_l}v)^T+\nabla \dv^\perp v.
\end{split}
\ee

\end{itemize}
\item For a hypersurface $M^n\hookrightarrow \R^{n+1}$,
\begin{itemize}
\item for the Bochner Laplacian, we have formula \eqref{projectB}, which reads

\be\label{projectBR}
(\tilde \nabla^\ast \tilde \nabla v)^T =\nabla^\ast  \nabla v+s^2v+nH(\tn_Nv)^T-(\tilde \nabla_N\tilde\nabla_N v)^T+\nabla_{\tilde\nabla_N N}v,
\ee
and is equivalent to
\be\label{projectRBe2}
\begin{split}
(\tilde \nabla^\ast \tilde \nabla v)^T& =\nabla^\ast  \nabla v-\Ric v+nH[N,v]^T-(\tilde \nabla_N\tilde\nabla_N v)^T+\nabla_{\tilde\nabla_N N}v.
\end{split}
\ee
From the Hodge Laplacian we get another equivalent formula
  \be\label{projectBdefdiv}
\begin{split}
(\tilde \nabla^\ast \tilde \nabla v)^T= -\Delta_hv+2s^2v+nH(\tn_{N}v)^T-nHsv-(\tilde \nabla_N\tilde\nabla_N v)^T+\nabla_{\tilde\nabla_N N}v.
\end{split}
\ee 

\item For divergence free vector field we have
\be\label{projectBdefdivfreeR2}
\begin{split}
(\tilde \nabla^\ast \tilde \nabla v)^T  &=\nabla^\ast  \nabla v-\Ric v+nH[N,v]-(\tilde \nabla_N\tilde\nabla_N v)^T+\nabla_{\tilde\nabla_N N}v\\
&=\nabla^\ast  \nabla v+s^2v+nH\tn_Nv-(\tilde \nabla_N\tilde\nabla_N v)^T+\nabla_{\tilde\nabla_N N}v\\
&= -\Delta_hv+2s^2v+nH\tn_{N}v-nHsv-(\tilde \nabla_N\tilde\nabla_N v)^T+\nabla_{\tilde\nabla_N N}v.
\end{split}
\ee

\item For the deformation Laplacian we obtain
\be\label{projectBdef2R}
\begin{split}
(\tilde \nabla^\ast \tilde \nabla v+\tilde \nabla \widetilde \dv v)^T  &=\nabla^\ast  \nabla v-\Ric v+ \nabla   \dv v +nH[N,v]^T\\
&\quad-(\tilde \nabla_N\tilde\nabla_N v)^T+\nabla_{\tilde\nabla_N N}v  +\nabla\dv^\perp v.
\end{split}
\ee
\end{itemize}
\end{itemize}
 \end{cor}
\begin{remark}
For the divergence free vector fields, the deformation Laplacian also agrees with the Hodge and Bochner Laplacian; in general, it is a separate formula, as in Corollary \ref{corProject}.  
 In  the case of a hypersurface, in the formula \eqref{projectBdefdivfreeR2} for divergence free vector fields, we get that the bracket  term is already tangential, without the projection, just as in Corollary \ref{corProject}.
\end{remark}
\begin{remark}
Comparing the formulas \eqref{projectBR1} and \eqref{projectRBe2} to  \eqref{BW} we see that \emph{if} the vector field is divergence free on $M$, the deformation tensor Laplacian is contained in the formula obtained from the Gauss formulas, but in general, there will be other terms that depend on the extension of $v$.  

If the vector field is not divergence free, then the connection with the deformation Laplacian and the projected operator is displayed in the formulas \eqref{Defproject} and \eqref{projectBdef2R}.  Here, similarly we also have additional terms that depend on the extension.  

We discuss this more in Section \ref{closerlook} after we review more background on Riemannian submanifolds.  
\end{remark}

%%%%%%%%%%%%%%%%%%%%%%%%%%%%%%%%%%%%%%%%%%%%%%%%%%%%%%%%%%%%
%%%%%%%%%%%%%%%%%%%%%%%%%%%%%%%%%%%%%%%%%%%%%%%%%%%%%%%%%%%%%%
 % % 
%
 
%

 \subsubsection{Lie derivative formula for a surface of revolution}
  
Using Corollary \ref{corProjectE} we can extend the formula in terms of the Lie derivatives of a $1$-form on an embedded ellipsoid as it was obtained in \cite{CCY_22} to a surface of revolution in $\R^3$.  To arrive at the formula we assume that 
 every line through the origin intersects the surface of revolution transversally.  This is assumed so the metric is invertible.
 
We have the following theorem.
\begin{thm}\label{thm2} Let $S$ be a surface of revolution in $\R^3$ such that  every line through the origin intersects S transversally.  Let $v$ be a divergence free vector field on $M$, $p \in M$, and $v$ be an extension of $v$ to a neighborhood of $p$ in $\R^{3}$, still denoted by $v$, and that is divergence free on $\R^3$.  Then

\begin{equation}\label{formulaSR}
\begin{split}
\iota_{S}^* \Big \{ -\triangle v^\flat \Big \} & =   -\triangle_{S}\left( \iota_{S}^* v^\flat\right) -  \iota_{S}^* \left\{\mathcal{L}_N\mathcal{L}_N v^\flat\right\}  + ( \kappa_1-\kappa_2 )\iota_{S}^*\left\{ \mathcal{L}_{N} v^\flat\right\} 
+ \frac 1{\ag^2} \iota_{S}^* \Big \{ \mathcal{L}_{\mathrm{Y}}  v^\flat  \Big\}\\
% \sqrt{K_{S}}
&\qquad+ 2(\kappa_2-\kappa_1)(\mathcal L_{N} v^\flat)_1 E^1-2\left(\frac{E_1(\ag)}{\ag}\right)^2v_1E^1,
\end{split}
\end{equation}
where
 \begin{itemize}
\item $v^\flat$ is $1$-form obtained by the musical isomorphism,
\item $\mathfrak i^\ast_S$ is the pullback by the inclusion map $\mathfrak i_S: S\hookrightarrow \R^3$,
\item $-\triangle$ is the (Hodge) Laplacian on $\R^3$,
\item  $ -\triangle_{\mathrm{S}}$ is the Hodge Laplacian on $S$,
\item $\mathcal L$ denotes the Lie derivative,
\item $N$ is the choice of the unit normal given by $N=\frac{\nabla\rho}{\ag}$,
\item  $\kappa_1, \kappa_2$ are the principal curvatures of S with respect to $\R^3$,
\item $\rho$ is the defining function as given in \eqref{parametrize},
%\item $Y=\rho \partial_\rho$;
\item $E^1$ is the $1$-form on $S$ dual to  $E_1$, which is a unit vector field in the direction along the meridians on $S$,
\item $Y= \ag E_1(\abs{\nabla\rho})E_1.$
 \end{itemize}

\end{thm}

\begin{remark}
In \cite{CCY_22} we worked with the Hodge Laplacian of a $1$-form in 3D.  Because all the Laplacians are equivalent on the Euclidean space by \eqref{BW}, instead of working with the Hodge Laplacian on the left hand side, we can work with the Bochner Laplacian of a vector field; this will be used in the proof of the formula.  We further show there is no ambiguity in this choice in Section \ref{commuting}.
\end{remark}
\begin{remark}
The existence of vector fields that are divergence free both on a surface and $\R^3$ has been addressed in \cite{CCY_22}.  The discussion was carried out for the case of the ellipsoid, but it extends here as well. The main idea is that such vector fields are as natural for the surface of revolution considered here as they are for the sphere.
\end{remark}
In order to connect formula \eqref{formulaSR} to a formula obtained in \cite{CCY_22}, we have the following corollary.

\begin{cor}\label{corthm2} With the same setting as in Theorem \ref{thm2} we have
\begin{equation}\label{formulaSR2}
\begin{split}
\mathfrak{i}_{\mathrm{S}}^* \Big \{ -\triangle v^\flat \Big \} & =   -\triangle_{\mathrm{S}}\Big ( \mathfrak{i}_{\mathrm{S}}^* v^\flat \Big ) + \mathcal{E} ( v^\flat )
+ \frac 1{\ag^2}\iota_E^* \Big \{ \mathcal{L}_{\mathrm{Y}}  v^\flat   \Big \}
% \sqrt{K_{\mathrm{E}}}
+\frac 2{\ag}(\kappa_2-\kappa_1)\mathfrak{i}_{\mathrm{S}}^* \Big \{ (\mathcal L_{\nabla \rho} v^\flat)_1 \Big \} E^1,
\end{split}
\end{equation}
where
 $\mathcal{E}$ is an operator given by
\[
\mathcal{E} =\iota_E^* \Big \{ -\mathcal{L}_{\nabla \rho} \Big ( \frac{1}{|\nabla \rho|^2}  \mathcal{L}_{\nabla \rho}     \Big ) +\Big( \frac{\kappa_1-\kappa_2}{\ag}-\frac{N(\ag)}{\ag^2}\Big)\mathcal{L}_{\nabla \rho} \Big \} .
 \]
\end{cor}
\begin{remark}
What we find interesting in the above two formulas is how they accentuate the behavior of the vector field along the $E_1$ direction, which corresponds to the direction along the meridians on the surface of revolution.  This is especially interesting, when compared to the formula \eqref{projectRBe2}, where there is no clear distinction between different directions on the submanifold.
\end{remark}
As examples of the above formulas, we discuss the case of the sphere and the ellipsoid in Section \ref{examples}.

We connect the differential geometry motivation with fluids motivation by the following final remarks. How can directional derivatives be defined on a submanifold?  By taking the extrinsic covariant derivative, and then projecting back to the submanifold (See e.g., \cite[p.86]{Lee_RG}).
This idea of projecting/restricting an extrinsic operator connects the differential geometry with fluids: obtaining formulas such as the ones in this article can lead to obtaining an intrinsic Laplacian-type operator on a submanifold to serve as a viscosity operator for a fluid problem.  At this point, we believe that the exact form of the Laplacian will depend on the physical problem at hand.  The benefit of having different formulas, which while equivalent, emphasize different terms, could be helpful in a study of boundary value problems, where some of those terms are assumed to be zero on the boundary.   We will investigate this further in a forthcoming article, where we will look at applications of these formulas and specific examples of submanifolds. 

The paper is organized as follows.  In Section \ref{prelim} we review the needed background from Riemannian geometry. Then, in Section \ref{closerlook}, we have a more in depth discussion of the Gauss formulas.  Next, we setup the notation and formulas for the surface of revolution result.  In Section 3 we prove the main Gauss formulas, and in Section 4 we show Theorem \ref{thm2} and its corollary.  The last section is devoted to various examples.

\section*{Acknowledgements}
Chi Hin Chan is funded in part by a grant from the Ministry of Science and Technology of Taiwan (109-2115-M-009 -009 -MY2). This work was partially completed while Chi Hin Chan was working as a Center Scientist at the National Center for Theoretical Science of Taiwan R.O.C. 

Magdalena Czubak is funded in part by a grant from the Simons Foundation \# 585745.

\section{Preliminaries}\label{prelim}

Here we establish notation and gather the necessary tools.  We begin with the terminology related to an embedded hypersurface, and then generalize the discussion to embedded submanifolds for any $\codim \geq 1$.  We then move on to a surface of revolution embedded in $\R^3$.   

In general, we sum over repeated indices.   When there might be some confusion, we include a sum.

The following setup is based on \cite[Chapter 8]{Lee_RG}.  Let $(\tM, \tg)$ be a Riemannian manifold, and $(M,g)$ be an embedded Riemannian submanifold with the induced metric $g$, i.e.,
\[
g=\iota^\ast \tg,
\]
where $\iota: M \hookrightarrow \tM$ is the inclusion map.

Let $\mathfrak X(M)$ denote the smooth vector fields on $M$.  Then if $X, v \in \mathfrak X(M)$, $X, v$ can be extended to be vector fields in a neighborhood of $M$ in $\tM$ (e.g., \cite[p. 384, Ex. A.35]{Lee_RG}).  Then the following (Gauss) formula holds on $M$
\be\label{gauss}
\tn_X v=\nabla_X v+\two(X,v),
\ee
where $\two : \mathfrak X(M) \times \mathfrak X(M) \to \Gamma (NM)$ denotes the second fundamental form, which sends two vectors fields on $M$ to smooth sections of the normal bundle of $M$.  See \cite[Chapter 2]{Lee_RG} for more on the normal bundle.

We record some properties of the second fundamental form.
\begin{prop}\cite[Prop. 8.1]{Lee_RG}\label{prop2}.  Let $X, Y \in \mathfrak X(M)$, then the second fundamental form satisfies the following properties.
\begin{itemize}
\item $\two(X,Y)$ is independent of the extensions of $X$ and $Y$.
\item $\two$ is bilinear over smooth functions on $M$.
\item $\two(X,Y)=\two(Y,X)$.
\item The value of $\two(X,Y)$ at $p \in M$ depends only on $X$ and $Y$ at $p$.
\end{itemize}
\end{prop}
\subsection{Hypersurfaces and the shape operator}\label{so_setup}
We now specialize to an embedded hypersurface.  Let $N$ be a choice of a unit normal, $N\in \Gamma(NM)$.  We can then introduce a scalar second fundamental form, given by
\be
h(X,v)=\tilde g(N, \two(X,v)).
\ee
It follows
\be\label{scalar2}
\two(X,v)=h(X,v)N.
\ee
We further define the shape operator
\[
s:  \mathfrak X(M) \to  \mathfrak X(M),
\]
by
\be\label{shape_def}
g(s X, v)=h(X,v).
\ee
The eigenvalues of $s$, $\kappa_i$ are called the principal curvatures, and at each $p\in M$, we have an orthonormal basis $\{E_i\}$ of eigenvectors.  Such basis exists by the self-adjoint property of $s$ \cite[Prop. 8.16]{Lee_RG}.

We write down the Weingarten equation for a hypersurface
\be\label{Weq}
sX=-\tilde \nabla_X N, 
\ee
as well as the Codazzi equation 
\be\label{PMC}
\tilde R(W, X, Y, N)=\nabla h (Y, X, W)-\nabla h (Y, W, X), \quad W, X, Y \in  \mathfrak X(M) .
\ee
Let $p\in M$, and $\{E_i\}$ be an orthonormal basis of the eigenvectors of $s$ for $T_pM$.   We can extend $\{E_i\}$ to be a local ON frame in some neighborhood $U$ of $p$.  We note that in general, we cannot assume that $\{E_i\}$ are eigenvectors of $s$ for all $q\in U$.  This has to do with a possible existence of umbilical points. See for example \cite[pp 123--124]{Clelland_book}.  However, it is sufficient for us that $\{E_i\}$ are eigenvectors of $s$ at $p$.

Next, $\{E_i, N\}$ can be extended locally to mutually orthogonal unit vector fields on $\tilde M$.  Then $\{E_i, N\}$ form a local orthonormal frame for $\tilde M$, sometimes referred to as an adapted frame.  Let $E_{n+1}=N$. If we write $E_\alpha$, we mean $\alpha\in\{1, \dots, n+1\}$, and $E_i$ means $i\in\{1, \dots, n\}$. Same holds for any other Greek or Roman index, respectively.

\subsection{Submanifolds with $\codim \geq 1$}\label{Subk}
We now extend the above discussion to $M^n\hookrightarrow \tilde M^{n+k}$ for any $k\geq 1$.  

In this case,  let $\{N_l\}, \ l=1, \dots, k$, be an orthonormal frame for the normal bundle, defined in a neighborhood of $p\in M$.

For each $N\in \Gamma(NM)$ (not necessarily a unit vector field) we can define a Weingarten map
\[
W_{N}:\Xm{M}\to \Xm{M}
\]
by
\be\label{Wm}
 g(W_N(X),Y)=\tilde g(N, \two(X,Y)).
\ee
The Weingarten equation then becomes
\be\label{Weq2}
W_N X=-(\tilde \nabla_X N)^T,
\ee
where $T$ denotes the projection onto $T_pM$.
We now have a simple lemma.
\begin{lemma}\label{s2L}
Let $N$ be a fixed normal $N$ defined in an open neighborhood in $M$.  If $\{E_i\}$ is any local ON frame on $M$, then
\be\label{s2}
W_{N}W_Nv=\delta^{ij}g(W_Nv, E_j) W_N E_i.
\ee
We also have
\be\label{s2l}
\sum_{l=1}^kW_{N_l}W_{N_l}v=\delta^{ab}\delta^{ij}\tg\big(\two(E_a, E_i),\two(v, E_b)\big) E_j.
\ee

\end{lemma}
\begin{proof}

To see \eqref{s2}, compute using the ON property of the frame, \eqref{Wm} and bilinearity of $\two$ as follows
\begin{align*}
W_{N}W_Nv&=\sum_ig(W_N W_N v, E_i)E_i\\
&=\sum_i\tilde g(N, \two ( W_N v, E_i))E_i\\
&=\sum_{ij}\tilde g\big(N, \two \left( g(W_Nv,E_j)E_j, E_i\right)\big)E_i\\
&=\sum_{ij} g(W_Nv,E_j)\tilde g(N, \two(E_j,E_i))E_i\\
&=\sum_{ij} g(W_Nv,E_j)g(W_N E_j, E_i)E_i\\
&=\sum_{j} g(W_Nv,E_j)W_N E_j,
\end{align*}
as needed.
Next, for \eqref{s2l} consider
\begin{align*}
\delta^{ab}\delta^{ij}\tg\big(\two(E_a, E_i),\two(v, E_b)\big) E_j&=\sum_{l}\delta^{ab}\delta^{ij}\tg\Big(\tg\big(\two(E_a, E_i),N_l\big)N_l,\two(v, E_b)\Big) E_j\\
&=\sum_{l}\delta^{ab}\delta^{ij}\tg\big(\two(E_a, E_i),N_l\big)g(W_{N_l}v, E_b) E_j\\
&=\sum_{l}\delta^{ab}\delta^{ij}g(W_{N_l}E_a, E_i)g(W_{N_l}v, E_b) E_j\\
&=\sum_{l}\delta^{ab}g(W_{N_l}v, E_b)W_{N_l}E_a,
\end{align*}
then \eqref{s2l} follows by \eqref{s2}.
\end{proof}
It is helpful to introduce the normal connection $\nabla^\perp: \Xm{M}\times \Gamma(NM)\to \Gamma(NM)$.  The normal connection is defined by
\be\label{nablap}
\nabla_X^\perp N=(\tn_XN)^\perp.
\ee
We also introduce the mean curvature vector, $\vH$.  In the case of a hypersurface, the scalar mean curvature $H$ is the average of the principal curvatures, the average of the trace of the shape operator.  The mean curvature vector is the generalization obtained by taking the trace of the second fundamental form (see e.g.,  \cite[p.258]{GHL})
\be\label{vH}
\vH=\frac 1n \delta^{ij}\two(E_i,E_j).
\ee 

Finally, we define a connection for bilinear maps over $C^\infty(M)$ taking values in the normal bundle.  More precisely, let $B\to M$ be a vector bundle 
\begin{align}
B= \bigsqcup_{p\in M} B_p,
\end{align}
where
\[
B_p=\{ A: T_pM\times T_p M\to N_p M \ |   \  A \  \mbox{is bilinear} \}.
\]
Then a connection in $B$, $\nabla^B: \mathfrak X(M)\times \Gamma(B)\to \Gamma (B)$ can be defined by
\be\label{nB}
(\nabla^B_X A)(Y,Z)=\nabla^\perp_X(A(Y,Z))-A(\nabla_XY, Z)-A(Y,\nabla_XZ),
\ee
where $A\in \Gamma(B)$ and $X, Y, Z\in  \mathfrak X(M)$.  This connection can be found in \cite[p.231]{Lee_RG}, and it applies to the second fundamental form $\two$.
Using $\nabla^B$ we can state the general form of the Codazzi equation 
\be\label{CodazziG}
(\tilde R(W, X)Y)^\perp=(\nabla^B_W \two)(X,Y)-(\nabla^B_X \two)(W,Y).
\ee
\subsection{Orthogonal divergence and orthogonal gradient}\label{orthdg}
It is natural to also introduce orthogonal divergence and orthogonal gradient as follows.  We first recall the following notation.  Let $(N, h)$ be a Riemannian manifold.  In a local frame $\{F_i\}$ on $N$, we use the notation
\be\label{best_notation}
\nabla_i v^j:=F_i(v^j)+\Gamma_{ik}^j v^k,
\ee
where $\Gamma_{ik}^j$ are the Christoffel symbols defined by
\[
\nabla_{F_i}F_k=\Gamma_{ik}^j F_j.
\]
So then if $X=X^iF_i$, and $Y=Y^jF_j$,
\be\label{obvious}
\nabla_X Y=X^i \nabla_i Y^j F_j.
\ee
Summing over repeated index, we can write then the divergence of $v$ as
\[
\dv v=\nabla_i v^i.
\]
Next, the divergence on $\tilde M$, in an ON frame can be written as
\be\label{diveq}
\widetilde \dv v=\tn_\alpha v^\alpha=\sum_\al\tg(\tn_{E_\al}v, E_\al).
\ee
We then define
\be\label{divperp}
\dv^\perp v=\tn_{n+l}v^{n+l},
\ee
on $M$.  Similarly if $f$ is smooth on $\tilde M$, then the gradient of $f$ can be written as
\[
\tn f=\sum_\al\tg(\tn f, E_\al)E_\al.
\]
Define
\be\label{gradperp}
\nabla^\perp f=\sum_{l=1}^k\tg(\tn f, N_l)N_l.
\ee
We can show that these definitions are independent of the choice of the adapted ON frame.  We also have the following lemma.
\begin{lemma}\label{lemdiv} Let $v$ be a smooth vector field on $\tilde M$.  Then on $M$ we have
\begin{align}
 \widetilde\dv v&=\dv v+\dv^\perp v,\label{divg}\\
\tn \widetilde\dv v&=\nabla \dv v+\nabla \dv^\perp v+\nabla^\perp\widetilde \dv v, \label{graddiv}
\end{align}
and if $v$ is divergence free on $M$ and on $\tilde M$, then $\dv^\perp v=0$ on $M$.
\end{lemma}
\begin{proof}
Formula \eqref{divg} follows from the Gauss formula \eqref{gauss} and \eqref{diveq}.  The fact that $\tn_i v^i=\nabla_i v^i$ could also be seen from computing directly using the Christoffel's symbols.  To see \eqref{graddiv}, use \eqref{divg} and the fact that for a function $f$
\be\label{usea}
E_i (f)=d f(E_i)=\tg (\tn f, E_i)=g(\nabla f, E_i).
\ee
Finally, if $v$ is divergence free on $M$ and on $\tilde M$, from \eqref{divg} we get $\dv^\perp v=0$ as needed.
 
\end{proof}
 
\subsection{Gauss formulas: a closer look}\label{closerlook}
Having introduced the main geometric objects and their relevant properties, we take a closer look at the Gauss formulas.

 We consider the general formulas for the Bochner Laplacian, for all codimension $k\geq 1$, and for the general ambient manifold $\tilde M$.  Analogous comments can be made about other formulas. The formulas for the Bochner Laplacian are the formula \eqref{GaussB1} 
\be\nonumber
\begin{split}
\tilde \nabla^\ast \tilde \nabla v &=\nabla^\ast\nabla v  +\sum_{l=1}^k W_{N_l}W_{N_l}v-nW_{\vH}v\\
&\quad+n[\vH, v]-\sum_{l=1}^k\tn_{N_l}\tn_{N_l}v+\sum_{l=1}^k\tn_{\tn_{N_l}N_l}v,\\
&\qquad -2\delta^{ij}\two(E_i, \nabla_{E_j}v)-\delta^{ij}(\tilde R(E_j,v)E_i)^\perp,
\end{split}
\ee
and \eqref{GaussB2} 
\be\nonumber
\begin{split}
\tn^\ast \tn v &=\nabla^\ast\nabla v+\sum_{l=1}^k W_{N_l}W_{N_l}v\\
&\quad+n \tn_{\vH}v-\sum_{l=1}^k\tn_{N_l}\tn_{N_l}v+\sum_{l=1}^k\tn_{\tn_{N_l}N_l}v\\
&\qquad -(\tr \nabla^B \two) (v)-2\delta^{ij}\two(E_i, \nabla_{E_j}v).
\end{split}
\ee
Both of the formulas are written so on the right hand side, the first line has only tangential terms, the third only normal terms, and the middle line can have both.  More precisely, for the middle line, to distinguish which terms are tangential and normal, we would need to apply the tangential and normal projections.  This is the first difference from the classical Gauss formula for the connection, which we state here again for convenience.
 \be\label{GaussC2}
 \tn_X v=\nabla_X v+\two(X, v).
 \ee

Here, the first term on the right hand side is the intrinsic covariant derivative, which is tangential, and the second fundamental form, which is normal.  Similarly, when we look at the Gauss formula for the curvature, we have for a general submanifold
\be\label{GaussCa}
\tilde R (W,X,Y,Z)=R(W,X,Y,Z)-\tilde g(\two (W,Z),\two(X,Y))+\tilde g(\two(W,Y), \two(X,Z)).
\ee
In this case, the first term on the right is the intrinsic curvature, and the rest involve the second fundamental form. 
 
However, what is important to note is that these classical formulas only consider vector fields that are given as tangential to the submanifold $M$ to begin with.  The issue is even more apparent, when we go back to the formula for the connection and instead of inputting tangential vector field $X$, we consider $\tn_N v$.  Then there is no clear formula to tell us how to break the term into something tangential with a clear meaning; all we can do is to use projections and write
\[
\tn_N v=(\tn_N v)^T+(\tn_N v)^\perp.
\]
Given an ON frame on $\tilde M$, $\{E_\alpha \}$, the Bochner Laplacian of a vector field can be viewed as involving taking two covariant derivatives in the direction of $E_\alpha$ (see Lemma \ref{key}).    For that reason, it should come as no surprise that we have the terms in the middle line in \eqref{GaussB1} and \eqref{GaussB2}.  The second and third term in the middle line are exactly from explicitly considering the covariant derivative in the normal directions.  The very first term arises when we consider covariant derivative in the direction of $\tn_{E_i}E_i$, which by the Gauss formula has tangential and normal components (see the proof of Theorem \ref{gaussThmM} for details).

We now discuss the dependence of the fomulas on the extension of the vector field $v$. In the case of the Gauss formula for the connection, the right hand side is independent of the extension.  Indeed, the normal part, by properties of the second fundamental form (see Proposition \ref{prop2}) depends only on $X, v$ at $p$.  The tangential part $\nabla_X v$ depends only on $X$ at $p$ and only on $v$ along some curve going through $p$, and in the direction of $X$ (see  e.g., \cite[Prop 4.26]{Lee_RG}).

Using this, we could show that purely normal terms in the third line of the formulas depend only on $v$ at $p$, and that the purely tangential terms, the terms in the first line, are independent of the extension to $\tilde M$.  Therefore, the terms in the second line are the ones that can affect both the tangential parts and the normal parts as they do depend on the extension.  We illustrate this further in the examples at the end of the paper.

Next, we would like to demonstrate that there is merit in including all the different formulas.
For the moment, we consider the formulas for a general Euclidean hypersurface $M$.  Looking for the ``correct" operator, we could declare that the correct operator is the one that we obtain by taking only the intrinsic terms that appear on the right hand side.  Given that we have a collection of different formulas, if we choose formula \eqref{projectBR}, then the correct operator would need to be the Bochner Laplacian.  On the other hand, if we choose formula \eqref{projectRBe2}, we would arrive at the deformation Laplacian, and finally, if we worked with \eqref{projectBdefdiv} we would produce the Hodge Laplacian.  Based on this, demanding that the correct operator is only the intrinsic one is not a well defined procedure.

Motivated by Gauss's Theorem Egregium (a combination of extrinsic terms producing an intrinsic ones, see also \eqref{genRicR} and \eqref{Ricci1}), we could relax the requirement, and say that we include the intrinsic terms and the extrinsic ones but that are independent of the extension of $v$.  This gives a unified result, which is
\be\label{intrext}
\nabla^\ast\nabla v+s^2v=\nabla^\ast\nabla v-\Ric v+nHsv,
\ee
and similarly for higher codimensions
\be
\tilde \nabla^\ast \tilde \nabla v =\nabla^\ast\nabla v  +\sum_{l=1}^k W_{N_l}W_{N_l}v=\nabla^\ast\nabla v-\Ric v +nW_{\vH}v.
\ee

It is interesting to point out that in the case of the sphere embedded in $\R^3$, \eqref{intrext} is exactly the Hodge Laplacian (see Example \ref{sphereEx}).

\subsection{Surface of revolution}
Consider the set
\[
R=\{ (r, z): r>0, z\in \R \}\subset \R^2,
\]
and let $C$ be a smooth, embedded $1$-submanifold of $R$.  Then the surface of revolution, denoted by $S$, is an embedded smooth surface in $\R^3$ given by
\[
S=\{ (x, y, z): (\sqrt{x^2+y^2}, z)\in C \},
\]
with $C$ being called the generating curve of $S$ (see e.g., \cite{Lee_RG}).

Now, let $p=(x_0, y_0, z_0)\in S$, $I\subset \R$ open, and let $$\gamma: I\to R$$ be a local parametrization of $C$, $\gamma(t)=(a(t), b(t))$, such that $$\gamma(t_0)=(\sqrt{x_0^2+y_0^2}, z_0)\in C,$$ for some $t_0\in I$.  We also suppose
$\dot \gamma=1$ for all $t\in I$, where $\dot \gamma$ denotes the (time) derivative of $\gamma$.
 
Using $\gamma$, we can now define the following parametrization of a neighborhood of $p\in S$ in $\R^3$
\[
\Phi:  (0,\infty)\times I \times (-\pi,\pi) \to \R^3,
\]
where
\be\label{parametrize}
\Phi(\rho, t, \theta)=(\rho a(t)\cos\theta, \rho a(t)\sin\theta, \rho b(t)).
\ee
When $\rho=1$, we obtain a parametrization of $S$, $\bar \Phi:  I \times (-\pi,\pi) \to \R^3$, $\bar \Phi(\cdot,\cdot)=\Phi(1, \cdot, \cdot).$

The Euclidean metric on $\R^3$ is given by a block diagonal matrix
\be\label{Rmetric}
\left(\begin{array}{ccc}
   g_{\rho \rho}&g_{\rho t}&0\\
   g_{t\rho }&g_{tt}&0\\
 0&0&g_{\theta\theta}\\
   \end{array}\right),
\ee
where
\begin{align}
g_{\rho\rho}=a^2+b^2, \quad g_{\rho t}&=g_{t\rho}=\rho(\dot a a+\dot b b),\label{g1}\\
g_{tt}=\rho^2, \quad g_{\theta\theta}&=\rho^2 a^2.\label{g2}
\end{align}
We need that the curve is parametrized so that
\be\label{cc1}
(b\dot a-a\dot b)^2\neq 0\ \mbox{for all} \ t \in I.
\ee
This will guarantee that the metric as defined above is invertible (see \eqref{RmetricI} below).
\begin{remark}
We can observe that condition \eqref{cc1} is equivalent to  requiring that every line through the origin intersects $S$ transversally.
\end{remark}
Once \eqref{cc1} holds, without loss of generality, we can assume
\[
f(t):=b\dot a-a\dot b>0 \ \mbox{for all} \ t \in I.
\]
The inverse metric of \eqref{Rmetric} is
\be\label{RmetricI}
\left(\begin{array}{ccc}
   g^{\rho \rho}&g^{\rho t}&0\\
   g^{t\rho }&g^{tt}&0\\
 0&0&g^{\theta\theta}\\
   \end{array}\right),
\ee
\begin{align}
g^{\rho\rho}&=\frac{1}{f^2},
 \quad\qquad g^{\rho t}=g^{t \rho}=-\frac{\dot a a+\dot b b}{f^2\rho},\label{grr}\\
g^{tt}&=\frac{a^2 +b^2}{f^2\rho^2}, \quad g^{\theta\theta}=\frac1{a^2\rho^2}.
\end{align}

The metric on $S$ is
\be\label{Smetric}
\left(\begin{array}{cc}
   g_{tt}&0\\
   0&g_{\theta\theta}\\
   \end{array}\right),
\ee
with $g_{tt}, g_{\theta\theta}$ as in \eqref{Rmetric} when $\rho=1$.

\subsection{Dictionary for the Lie derivative formula}\label{dictionary}
Here  we connect the vector field setup with the surface of revolution as introduced in the previous sections.

We will be working in an ON frame consisting of the eigenvectors of the shape operator at $p$ in $S$.  In the case of the surface of revolution, the frame can be defined by
\[
E_1=\frac{\partial_t}{\rho}, \quad E_2=\frac{\partial_\theta}{a\rho},\quad E_3=N={f\nabla \rho},
\]
where $N$ is used to denote the choice of the normal vector, and where we used that $$\ag^2=g(\nabla \rho, \nabla \rho)=g(\dd \rho, \dd \rho)=g^{\rho\rho},$$ and \eqref{grr}.
It is also useful to write down
\be\label{drho}
\partial_\rho=\frac N{\abs{\nabla \rho}}-\rho\frac{g^{\rho t}}{\ag^2}E_1.
\ee
We need Christoffel symbols.  Working in an ON frame, it can be shown (see e.g., \cite[p.124]{Lee_RG})
\be\label{GammaFrame}
\Gamma_{\alpha\beta}^\gamma=\frac 12(c^\gamma_{\alpha\beta}-c^\beta_{\alpha\gamma}-c^\alpha_{\beta\gamma}),
\ee
where $c^{\tilde \gamma}_{\tilde \alpha \tilde \beta}$ are defined through the use of the Lie bracket by
\be\label{cbracket}
[E_{\tilde \alpha}, E_{\tilde \beta}]=c^{\tilde \gamma}_{\tilde \alpha \tilde \beta}E_{\tilde \gamma}.
\ee
From this we can also see that
\be\label{nonsymmetric}
\Gamma_{\alpha\beta}^\gamma=\Gamma_{\beta\alpha}^\gamma+c_{\alpha\beta}^\gamma.
\ee
Computing, using \eqref{drho}, we have
\begin{align*}
[E_1, N]=\left(\frac \ag \rho+\rho \frac{E_1(g^{\rho t})}{\ag}-2\rho \frac{g^{\rho t}}{\ag^2}E_1(\ag)\right)E_1+\frac{ E_1(\abs{\nabla\rho})}{\ag}N.
\end{align*}
Hence
\be\label{c313}
c^2_{13}=0, \quad c^3_{13}=\frac{E_1(\abs{\nabla\rho})}{\ag}=\frac{a\ddot b-b\ddot a}{\rho f},
\ee 
and another, longer computation shows
\be\label{c13}
c^1_{13}=-\frac{\ddot aa+\ddot bb}{\rho f}.
\ee
Similarly,
\begin{align}
[E_2, N]=c^2_{23}E_2=-\frac{\dot b}{a\rho}E_2\label{c223},
\end{align}
and
\begin{align}
[E_1, E_2]=c^2_{12}E_2=-\frac {\dot a}a E_2.\label{c212}
\end{align}
 It follows (we write down only the Christoffel symbols that are needed in this paper)
\begin{align}
\Gamma^1_{31}&=0=\Gamma^2_{31},\quad \Gamma_{31}^3=c^3_{31},\label{c1}\\
\Gamma^1_{32}&=\Gamma^2_{32}=\Gamma_{32}^3=c^3_{32}=0,\nonumber\\
\Gamma^1_{33}&=c^3_{13},\quad\Gamma^2_{33}=0=\Gamma^3_{33},\nonumber\\
\Gamma^1_{11}&= \Gamma_{11}^2=0,\quad \Gamma^3_{11}=c^1_{31},\nonumber\\
\Gamma^1_{12}&= \Gamma_{12}^2=\Gamma^3_{12}=0.\nonumber
\end{align}

This, together with \eqref{nonsymmetric}, implies
\begin{align}
\tn_N E_1&=\Gamma^\alpha_{31}E_\alpha=c_{31}^3N,\quad \tn_N E_2=\Gamma^\alpha_{32}E_\alpha=0,\quad\tn_NN=c^3_{13}E_1,\label{tnN}\\
 \tn_{E_1} N&=\Gamma^\alpha_{13}E_\alpha=c_{13}^1E_1,\quad
\tn_{E_2} N=\Gamma^\alpha_{23}E_\alpha=c_{23}^2E_2,\label{Weq1}
\end{align}
and
\be\label{t11}
\tn_{E_1}E_1=c_{31}^1N,
\ee
\be\label{t12}
\tn_{E_1}E_2=0.
\ee
We note that from \eqref{Weq} and \eqref{Weq1} at $p\in S$ we have
\be\label{kappas}
c_{31}^1=\kappa_1=\frac{\ddot aa+\ddot b b}{f},\quad c_{32}^2=\kappa_2=\frac{\dot b}a,
\ee
where $\kappa_i$ denotes the principal curvature.
 
Next, for our vector field $v=v^\alpha E_\alpha$, if $v$ is tangential to $M$, in this frame we have
\be\label{v0}
v^3=0\ \mbox{at}\ p\in S.
\ee
It follows
\be\label{tnv}
\tn_Nv=N(v^\alpha)E_\alpha+v^1c^3_{31}N\quad\mbox{at}\ \ p\in S,
\ee
and if the vector field is both divergence free on $\R^3$ and $S$ (see below for the verification),
\be\label{divfree}
N(v^3)+v^1\Gamma^3_{31}=N(v^3)+v^1c^3_{31}=0\quad\mbox{at}\ \ p\in S,
\ee
so in particular
\be\label{useful_div}
\tilde g(\tn_N v, N)\rvert_{p\in S}=0.
\ee
We verify \eqref{divfree} now.  If $v$ is divergence free on $\R^3$, then
\[
0=\tn_\alpha v^\alpha=\tn_i v^i+\tn_3 v^3,
\]
and because $v$ is divergence free on $S$,
\[
0=\nabla_i v^i=E_i(v^i)+\Gamma_{ij}^iv^j=E_i(v^i)+\Gamma_{i\al}^iv^\al=\tn_i v^i,
\]
since $v^3=0$ and because $\Gamma_{mn}^l$ are the same on $S$ and $\R^3$ when evaluated at $p\in S$ (see \eqref{GammaFrame}).

\subsection{Lie derivative formulas}
Here we record some useful formulas.  The following statement must be well-known in the literature, but we have not seen the proof, so we give a short proof here for completeness.

\begin{lemma}\label{usefulg}
Let $(\tilde M,\tilde g)$ be a Riemannian manifold, $X, Y\in  \Xm{\tilde M}$ and $\o$ be a $1$-form on $\tilde M$, and let $\ip{\cdot,\cdot}$ denote the duality pairing between a $1$-form and a vector field.  We have
\begin{align}
\ip{\mathcal L_X\o, Y}&=\ip{\tn_X \o, Y}+\ip{\o, \tn_{Y}X}\label{useful1a}\\
&=\tilde g(\tn_X \o^\sharp, Y)+\tilde g(\o^\sharp, \tn_{Y}X)\label{useful2}
\end{align}
where $\tn$ is the Levi-Civita connection on $\tilde M$.
\end{lemma}
\begin{proof} We present a coordinate-independent proof.  From the Cartan formula we have
\[
\mathcal L_X\o=\iota_X\dd \o+ \dd(\iota_X \o),
\]
where $\iota$ is an interior multiplication.

It follows
\begin{align*}
\ip{\mathcal L_X\o, Y}&=\ip{\iota_X\dd \o,Y}+\ip{ \dd(\iota_X \o),Y}\\
&=\dd \o(X,Y)+\ip{\dd(\o(X)),Y}\\
&=\dd \o(X,Y)+Y(\o(X)).
%&=\dd \o(X,Y)+Y\tilde g(\o^\sharp, X).
\end{align*}
Next, we use that for a $1$-form $\o$, it holds that (see e.g., \cite[p.369]{Lee_RG}),
\[
\dd \o(X,Y)=X(\o(Y))-Y(\o(X))-\o([X,Y]).
\]
Hence, by definition of the Levi-Civita connection,
\begin{align*}
\ip{\mathcal L_X\o, Y}&=X(\o(Y))-\o([X,Y])\\
&=X\tilde g(\o^\sharp, Y)-\tilde g(\o^\sharp, \tn_XY-\tn_YX),\\
&=\tilde g(\tn_X \o^\sharp, Y)+\tilde g(\o^\sharp, \tn_{Y}X),
\end{align*}
which shows \eqref{useful2}.  Then \eqref{useful1a} follows by definition of the $\sharp$ operator, and the definition of the covariant derivative of a $1$-form. 
\end{proof}
The following corollary gives us formulas for the $\alpha$ component function of the Lie derivative of a form. 
\begin{cor}\label{useful}
Let $(\tilde M,\tilde g)$ be a Riemannian manifold, $X\in  \Xm{\tilde M}$, $\o$ be a $1$-form on $\tilde M$, and $\{E_\alpha\}$ be an ON frame on $\tilde M$. 
Denote 
\[
(\mathcal L_X\o)_\al=\ip{\mathcal L_X\o, E_\alpha}.
\]
We then have
\be\label{useful1}
(\mathcal L_X\o)_\al=g(\tn_X \o^\sharp, E_\al)+g(\o^\sharp, \tn_{E_\al}X),
\ee
as well as
\be\label{relate}
%Let $\tilde \o_\al$ denote $\ip{\tilde \o, E_\al}$.  Then
(\mathcal L_X\o)_\al=[X,\o^\sharp]_\al+g(\tn_{\o^\sharp}X, E_\al)+g(\o^\sharp, \tn_{E_\al}X).
\ee
\end{cor}
\begin{remark}
While there is a nice relationship between the covariant derivative of a $1$-form and the covariant derivative of the corresponding vector field, namely
\be\label{special27}
\tn_X\omega =(\tn_X \omega^\sharp)^\flat,
\ee
and in fact \eqref{special27} is used sometimes to define the covariant derivative of a $1$-form, this formula no longer holds for the Lie derivatives, i.e., 
\be
\mathcal L_X \o\neq (\mathcal L_X \o^\sharp)^\flat,
\ee
in general.
This corollary shows they can be related through the use of equation \eqref{relate}.
\end{remark}
\begin{proof}
Equation \eqref{useful1} follows immediately from Lemma \eqref{usefulg}, and \eqref{relate} follows from \eqref{useful1}, and
\[
[X,\o^\sharp]_\al=\tilde g(\tn_X\o^\sharp, E_\al)-\tilde g(\tn_{\o^\sharp}X, E_\alpha).
\]
\end{proof}
A useful consequence of Corollary \ref{useful}, \eqref{useful_div} and \eqref{tnN} that we use in sequel is the following
\be\label{Lie3}
(\mathcal L_{N}v^\flat)_3=g(\tn_{N}v, N)+g(v,\tn_NN)=c^3_{13}v^1 \quad\mbox{at}\ \ p\in S.
\ee

Using Corollary \ref{useful}, we also derive
\begin{lemma}\label{Lies} Let $S$ be an embedded hypersurface in the Euclidean space, $\rho$ a locally defining function, $N=\frac{\nabla \rho}{\ag}$, $v\in \Xm{S}$, and let $s$ denote the shape operator defined in \eqref{shape_def}, Then
\be
((\mathcal L_{N}v^\flat)^\sharp)^T=[N,v]^T-2sv.
\ee
\end{lemma}
\begin{proof}
In this proof we implicitly assume we are evaluating at $p\in M$. First,  
\begin{align}
(( \mathcal L_{N}v^\flat)^\sharp)^T= (\mathcal L_{N}v^\flat)_i E_i,
\end{align}
and by \eqref{relate}
\[
(\mathcal L_{N}v^\flat)_i=[N,v]_i+ g(\tn_{v}N, E_i)+ g(v, \tn_{E_i}N).
\]
Next by \eqref{Weq}
\begin{align*}
&\tn_v N=v^\alpha \tn_{E_\al}N=v^j \tn_{E_j}N=-\sum_jv^j\kappa_jE_j,\\
&g(v, \tn_{E_i}N)=-g(v^\alpha E_\al, \kappa_i E_i)=-v^i\kappa_i,
\end{align*} 
and there is no sum in $i$ in the last line.
It follows
\[
(\mathcal L_{N}v^\flat)_i=[N,v]_i-2v^i\kappa_i,
\]
 as needed.
 \end{proof}

Another useful computation, using Corollary \ref{useful}, \eqref{Weq1}-\eqref{t12}, \eqref{v0} and the torsion free property, for $p\in S$, is
\begin{align}
f^2\ip{\mathcal L_{\frac{ c^3_{13}}{f^2}E_1} v^\flat, E_i}&=c^3_{13}g(\tn_{E_1}v, E_i)+f^2g(v, \tn_{E_i}\frac{c^3_{13}}{f^2}E_1)\nonumber\\
&=c^3_{13}E_1(v^i)+f^2\delta_{i1}E_i(c^3_{13}/f^2)v^1+c^3_{13}g(v, \tn_{E_i}E_1)\nonumber\\
&=c^3_{13}E_1(v^i)+f^2\delta_{i1}E_i(c^3_{13}/f^2)v^i+\delta_{i2}c^3_{13}c^2_{21}v^i\label{Liey}.
\end{align}

We finish Section \ref{prelim} with a brief discussion of pullbacks and projections.
\subsection{Projections versus pullbacks}\label{commuting}
The purpose of this section is to show that when working with an ON frame, it is easy to see the relationship of the projection and pullbacks to the appropriate $\flat$ and $\sharp$ operators.  More precisely, let $X\in \Xm{\tilde M}$ be given by
\[
X=X^\alpha E_\alpha,
\]
and now project onto $M$ to obtain
\[
X^{T_M}=X^i E_i.
\]
On the other hand, we can begin with $X^{\flat_{\tilde M}}$, so that
\[
X^{\flat_{\tilde M}}=X_\alpha E^\alpha,\quad X^\alpha=X_\alpha,
\]
and consider
\[
\iota^*_M(X^{\flat_{\tilde M}})=X_i E^i.
\]
Applying $\sharp_M$ gives
\[
(\iota^\ast_M(X^{\flat_{\tilde M}}))^{\sharp_M}=X_i E_i=X^{T_M},
\]
since in an ON frame $X_i=X^i$.

Equivalently
\[
\iota^*_M\circ \flat_{\tilde M}=\flat_M\circ T_M,
\]
which is also equivalent to
\[
\sharp_M\circ \iota^*_M= T_M\circ \sharp_{\tilde M}.
\]
 The following diagram summarizes this.
 \bigskip
 
  \begin{figure}
 \begin{center}
 \tikzcdset{every label/.append style = {font =\small}}
\begin{tikzcd}[row sep=50, column sep=50]
\Lambda^1(\tilde M) \arrow[r, shift left, "\sharp_{\tilde M}"]
\arrow[r, <-, shift right, swap, "\flat_{\tilde M}"]
%\arrow[dr, shift right=2]
 \arrow[d, shift left,"\iota^*_{M}"]
& \Xm{\tilde M} \arrow[d, shift left,"T_{M}"] \\
\Lambda^1(M)  \arrow[r, shift left, "\sharp_{M}"]
\arrow[r, <-, shift right, swap, "\flat_{M}"]
&\Xm {M}
%& \Xm{M}  
%\arrow[ul, shift left=4, "q_{quxfoo}"]

\end{tikzcd}
\caption{Vector fields versus forms.}

 \end{center}
 \end{figure}
 \bigskip
 
 This shows that we have freedom to choose if we rather work with a vector field or a form. 
 
 \section{Proofs of the main formulas}
 
 We start with the theorems that directly involve the vector fields and do not use the Lie derivative of a $1$-form.  First we have some preliminary notation and a lemma.
 
Let $(N, h)$ be a Riemannian manifold.
The Bochner Laplacian of a vector field $v$ is given by
\[
 \nabla^\ast \nabla v=-\dv \nabla v,
 \]
  where $\nabla$ is the Levi-Civita connection on $(N, h)$.  The Bochner Laplacian is sometimes called the connection Laplacian or the trace Laplacian.  
  
 We now have the following lemma, where we use \eqref{best_notation}.

\begin{lemma}\label{key}
Let $(N,h)$ be a Riemannian manifold, and $\{F_i\}$ be a local frame on $N$.  Then
\be
\nabla^\ast \nabla v=-h^{ik}\nabla_{F_i}\nabla_{F_k}v+h^{ik}\nabla_{\nabla_{F_i}{F_k}}v.
\ee
\end{lemma}
\begin{remark}
This lemma can follow from taking the trace of the identity in \cite[Prop. 4.21]{Lee_RG}.  However, because Lemma \ref{key} is key for our proof, we present a straightforward and self-contained proof below.
\end{remark}
\begin{proof}
By using \eqref{best_notation}, we have
\be\label{bochner1}
\nabla^\ast \nabla v=-\nabla^i \nabla_i v^j F_j.
\ee
This means the coordinate function in \eqref{bochner1} is given by
\begin{align}
-\nabla^i\nabla_i v^j&=-h^{ik}\nabla_k\nabla_i v^j\nonumber\\
&=-h^{ik}(F_k\nabla_i v^j+\Gamma^j_{k\ell}\nabla_i v^\ell-\Gamma^\ell_{ki}\nabla_\ell v^j),\label{bochner}
\end{align}
where we used a formula for a covariant derivative of a $(1,1)$ tensor (see e.g., \cite[Prop. 4.18]{Lee_RG}).

Now, use \eqref{obvious}, the properties of the connection, and consider
\begin{align*}
-h^{ik}\nabla_{F_k}\nabla_{F_i}v&=-h^{ik}\nabla_{F_k}(\nabla_i v^j F_j)\\
&=-h^{ik}(F_k\nabla_i v^j F_j+\nabla_i v^j \Gamma^\ell_{kj}F_\ell)\\
&=-h^{ik}(F_k\nabla_i v^j +\nabla_i v^\ell \Gamma^j_{k\ell})F_j.
\end{align*}
Comparing to \eqref{bochner}, we have
\[
\nabla^\ast\nabla v=-h^{ik}\nabla_{F_k}\nabla_{F_i}v+h^{ik}\Gamma^\ell_{ki}\nabla_\ell v^j F_j,
\]
and we observe that the second term can be rewritten as follows
\begin{align*}
h^{ik}\Gamma^\ell_{ki}\nabla_\ell v^j F_j&=h^{ik}\Gamma^\ell_{ki}\nabla_{F_\ell}v\\
&=h^{ik}\nabla_{\Gamma^\ell_{ki}F_\ell}v\\
&=h^{ik}\nabla_{\nabla_{F_k}F_i}v,
\end{align*}
as needed.
\end{proof}
\subsection{Proof of Theorem \ref{gaussThmM}}\label{proofofshapeop}
We are now ready to start deriving formulas \eqref{GaussB1} and  \eqref{GaussB2}.  Let $v\in \mathfrak X(M)$, $\dim M=n$, and extend it to be locally a vector field on $\tilde M$, $\dim \tilde M=n+k$.  By Lemma \ref{key}, and using the notation introduced in Section \ref{so_setup}, we have
\begin{align*}
\tn^\ast \tn v&=-\delta^{\al\bt}\tn_{E_\al}\tn_{E_\bt}v+\delta^{\al\bt}\tn_{\tn_{E_\al}{E_\bt}}v\\
&=-\delta^{ij}\tn_{E_i}\tn_{E_j}v+\delta^{ij}\tn_{\tn_{E_i}{E_j}}v-\sum_l\tn_{N_l}\tn_{N_l}v+\sum_l\tn_{\tn_{N_l}N_l}v\\
&=I+II-\sum_l\tn_{N_l}\tn_{N_l}v+\sum_l\tn_{\tn_{N_l}N_l}v,
\end{align*}
where we use $\sum_l$ to denote $\sum_{l=1}^k$.

By \eqref{gauss}, it follows
\begin{align*}
I&=-\delta^{ij}\tn_{E_i}\tn_{E_j}v\\
&=-\delta^{ij}\tn_{E_i}(\nabla_{E_j}v+\two(E_j, v))\\
&=Ia+Ib.
\end{align*}
Then, evaluating at $p\in M$, again by \eqref{gauss}  
\be\label{Ia}
Ia=-\delta^{ij}\tn_{E_i}\nabla_{E_j}v=-\delta^{ij}\nabla_{E_i}\nabla_{E_j}v-\delta^{ij}\two(E_i, \nabla_{E_j}v).
\ee

For $Ib$, we first observe that by definition $\two(E_j, v)\in \Gamma (NM)$.  It follows
\[
\two(E_j, v)=\sum_l \tg(N_l, \two(E_j, v)) N_l.
\]
So by definition of the Weingarten map, \eqref{Wm},
\[ 
\two(E_j, v)=\sum_l \tg(W_{N_l}v, E_j) N_l.
\]
Then 
\begin{align}
Ib&=-\delta^{ij}\tn_{E_i}\two(E_j,v)\nonumber\\
&=-\delta^{ij}(\tn_{E_i}\two(E_j,v))^T-\delta^{ij}\nabla^\perp_{E_i}\two(E_j,v)\nonumber\\
&=-\delta^{ij}\sum_l(\tn_{E_i} \tg(W_{N_l}v, E_j) N_l)^T-\delta^{ij}\nabla^\perp_{E_i}\two(E_j,v)\nonumber\\
&=-\delta^{ij}\sum_l\left(E_i   \tg(W_{N_l}v, E_j) N_l  + \tg(W_{N_l}v, E_j)\tn_{E_i} N_l\right)^T-\delta^{ij}\nabla^\perp_{E_i}\two(E_j,v)\nonumber\\
&=-\delta^{ij}\sum_l  \tg(W_{N_l}v, E_j)(\tn_{E_i} N_l)^T-\delta^{ij}\nabla^\perp_{E_i}\two(E_j,v)\nonumber\\
&=\sum_l W_{N_l}W_{N_l}v-\delta^{ij}\nabla^\perp_{E_i}\two(E_j,v),\label{Ib}
\end{align}
by \eqref{Weq2} and \eqref{s2}.

Next
\begin{align*}
II&=\delta^{ij}\tn_{\tn_{E_i}{E_j}}v\\
&=\delta^{ij} \left( \tn_{\nabla_{E_i}{E_j}}v+\tn_{\two(E_i,E_j)}v\right)\\
&=IIa+IIb.
\end{align*}
We have
\be\label{IIa}
IIa=\delta^{ij} \tn_{\nabla_{E_i}{E_j}}v=\delta^{ij} \left( \nabla_{\nabla_{E_i}{E_j}}v+\two(\nabla_{E_i}{E_j}, v)\right),
\ee
and 
\be\label{IIb}
IIb=\delta^{ij} \tn_{\two(E_i,E_j)}v=n\tn_{\vH}v,
\ee
where $\vH$ is the vectorial mean curvature defined in \eqref{vH}.
 
Combining \eqref{Ia}-\eqref{IIb}, we obtain

\begin{align*}
\tn^\ast \tn v &=-\delta^{ij}\nabla_{E_i}\nabla_{E_j}v-\delta^{ij}\two(E_i, \nabla_{E_j}v)+\sum_l W_{N_l}W_{N_l}v-\delta^{ij}\nabla^\perp_{E_i}\two(E_j,v)\\
&\quad+\delta^{ij} \left( \nabla_{\nabla_{E_i}{E_j}}v+\two(\nabla_{E_i}{E_j}, v)\right)+n \tn_{\vH}v\\
&\quad-\sum_l\tn_{N_l}\tn_{N_l}v+\sum_l\tn_{\tn_{N_l}N_l}v.
\end{align*}
Next, we apply Lemma \ref{key} to $(M, g)$. This gives
\begin{align*}
\tn^\ast \tn v &=\nabla^\ast\nabla v-\delta^{ij}\two(E_i, \nabla_{E_j}v) +\sum_l W_{N_l}W_{N_l}v-\delta^{ij}\nabla^\perp_{E_i}\two(E_j,v)\\
&\quad+\delta^{ij} \two(\nabla_{E_i}{E_j}, v)+n \tn_{\vH}v-\sum_l\tn_{N_l}\tn_{N_l}v+\sum_l\tn_{\tn_{N_l}N_l}v.
\end{align*}

We now gather the terms explicitly containing $\two$.  These are
\be\label{explicitTWO}
 -\delta^{ij}\two(E_i, \nabla_{E_j}v) -\delta^{ij}\nabla^\perp_{E_i}\two(E_j,v)
 +\delta^{ij} \two(\nabla_{E_i}{E_j}, v).
\ee

We compare them to $\nabla^B \two$ introduced in Section \ref{prelim}.  More precisely, we consider
\be\label{nablaBs}
(\nabla^B_{E_j}\two)(E_i, v)=\nabla_{E_j}^\perp(\two(E_i, v))-\two(\nabla_{E_j}E_i, v)-\two(E_i, \nabla_{E_j}v).
\ee

It follows, \eqref{explicitTWO} can be written as

\be\label{explicitNa}
-2\delta^{ij}\two(E_i, \nabla_{E_j}v)-\delta^{ij} (\nabla^B_{E_j}\two)(E_i, v).
\ee
We can think of the second term as the trace of $\nabla^B \two$ in the first and third component and write
\be\label{trB}
\delta^{ij} (\nabla^B_{E_j}\two)(E_i, v)=(\tr \nabla^B \two) (v).
\ee

Putting it all together we have
\be\label{im1}
\begin{split}
\tn^\ast \tn v &=\nabla^\ast\nabla v-2\delta^{ij}\two(E_i, \nabla_{E_j}v) +\sum_l W_{N_l}W_{N_l}v-(\tr \nabla^B \two) (v)\\
&\quad +n \tn_{\vH}v-\sum_l\tn_{N_l}\tn_{N_l}v+\sum_l\tn_{\tn_{N_l}N_l}v.
\end{split}
\ee
This is \eqref{GaussB2}.

Equivalently, we can use that since $\two$ is symmetric, then so is $\nabla^B_{E_j} \two$, and
\[
(\nabla_{E_j}^B\two)(E_i, v)=(\nabla_{E_j}^B\two)(v, E_i)=(\tilde R(E_j,v)E_i)^\perp+(\nabla^B_{v}\two)(E_j, E_i),
\]
where the last equality holds by the Codazzi equation \eqref{CodazziG}. Hence, \eqref{im1} can be also written as
\be\label{im2}
\begin{split}
\tn^\ast \tn v &=\nabla^\ast\nabla v-2\delta^{ij}\two(E_i, \nabla_{E_j}v) +\sum_l W_{N_l}W_{N_l}v-\delta^{ij}(\tilde R(E_j,v)E_i)^\perp-(\nabla^B_{v}\two)(E_j, E_i)\\
&\quad +n \tn_{\vH}v-\sum_l\tn_{N_l}\tn_{N_l}v+\sum_l\tn_{\tn_{N_l}N_l}v.
\end{split}
\ee
We can continue further by explicitly writing out $(\nabla^B_{v}\two)(E_j, E_i)$.
\begin{align*}
\tn^\ast \tn v &=\nabla^\ast\nabla v-2\delta^{ij}\two(E_i, \nabla_{E_j}v) +\sum_l W_{N_l}W_{N_l}v-\delta^{ij}(\tilde R(E_j,v)E_i)^\perp\\
&\quad -\delta^{ij}\nabla_{v}^\perp(\two(E_i, E_j))+2\delta^{ij}\two(\nabla_{v}E_i, E_j)\\
&\qquad +n \tn_{\vH}v-\sum_l\tn_{N_l}\tn_{N_l}v+\sum_l\tn_{\tn_{N_l}N_l}v\\
 &=\nabla^\ast\nabla v+2\delta^{ij}\two(E_i, [v, E_j]) +\sum_l W_{N_l}W_{N_l}v-\delta^{ij}(\tilde R(E_j,v)E_i)^\perp\\
&\quad -n\nabla_{v}^\perp \vH+n \tn_{\vH}v-\sum_l\tn_{N_l}\tn_{N_l}v+\sum_l\tn_{\tn_{N_l}N_l}v\\
%&\qquad +n \tn_{\vH}v-\sum_l\tn_{N_l}\tn_{N_l}v+\sum_l\tn_{\tn_{N_l}N_l}v\\
 &=\nabla^\ast\nabla v+2\delta^{ij}\two(E_i, [v, E_j]) +\sum_l W_{N_l}W_{N_l}v-\delta^{ij}(\tilde R(E_j,v)E_i)^\perp\\
&\quad -n\tn_{v}\vH+n(\tn_v\vH)^T+n \tn_{\vH}v-\sum_l\tn_{N_l}\tn_{N_l}v+\sum_l\tn_{\tn_{N_l}N_l}v\\
% &=\nabla^\ast\nabla v+2\delta^{ij}\two(E_i, [v, E_j]) +\sum_l W_{N_l}W_{N_l}v-\delta^{ij}(\tilde R(E_j,v)E_i)^\perp\\
%&\quad -n\tn_{v}\vH+n(\tn_v\vH)^T+n \tn_{\vH}v-\sum_l\tn_{N_l}\tn_{N_l}v+\sum_l\tn_{\tn_{N_l}N_l}v\\
&=\nabla^\ast\nabla v+2\delta^{ij}\two(E_i, [v, E_j]) +\sum_l W_{N_l}W_{N_l}v-\delta^{ij}(\tilde R(E_j,v)E_i)^\perp\\
&\quad -n[v,\vH]-nW_{\vH}v-\sum_l\tn_{N_l}\tn_{N_l}v+\sum_l\tn_{\tn_{N_l}N_l}v.
\end{align*}
This gives \eqref{GaussB1} after we show that one of the above terms can be simplified.  Namely
\[
2\delta^{ij}\two(E_i, [v, E_j])=-2\delta^{ij}\two(E_i, \nabla_{E_j}v),
\]
because it can be actually shown that
\be\label{surprise}
\delta^{ij}\two(E_i, \nabla_{v}E_j)=0.
\ee
This is easily seen in codimension one, using the frame of the eigenvectors at $p$.  The proof, in general, can proceed as follows.  By definition of $\two$
\begin{align*}
\two(\nabla_{v}E_i, E_i)&=(\tn_{\nabla_{v}E_i}E_i)^\perp\\
&=(v^j\tn_{\nabla_{E_j}E_i}E_i)^\perp\\
&=v^j\Gamma_{ji}^a\sum_l \Gamma_{ai}^{n+l}N_l.
\end{align*}
Next, for fixed $j$ and $l$, and summing with respect to $a$ and $i$, we have
\begin{align*}
4\Gamma_{ji}^a \Gamma_{ai}^{n+l}&=(c_{ji}^a-c_{ja}^i-c_{ia}^j)(c_{ai}^{n+l}-c_{a(n+l)}^i-c_{i(n+l)}^a)\\
&=-(c_{ji}^a-c_{ja}^i-c_{ia}^j)(c_{a(n+l)}^i+c_{i(n+l)}^a).
\end{align*}
where we used that the Lie bracket of tangential vector fields remains tangential, so it must be that $c_{ai}^{n+l}=0$.  Then splitting into two sums: $-(c_{ji}^a-c_{ja}^i-c_{ia}^j)c_{a(n+l)}^i$ and $-(c_{ji}^a-c_{ja}^i-c_{ia}^j)c_{i(n+l)}^a$, using that $c^j_{ia}=c^j_{ai}$, switching the roles of $a$ and $i$, say in the second sum, and then summing with respect to $a$ and $i$ we get zero as needed.

\subsubsection{Reducing to a hypersuface}
We now show how, in the case of a hypersurface, \eqref{GaussB1} reduces to \eqref{b4projectB}, and \eqref{GaussB2} reduces to \eqref{b4projectB2}.  For convenience we recall  \eqref{b4projectB} 

\be\label{b4projectBa}
\begin{split}
\tilde \nabla^\ast \tilde \nabla v &=\nabla^\ast  \nabla v+s^2v+nH\tn_Nv-\tilde \nabla_N\tilde\nabla_N v+\nabla_{\tilde\nabla_N N}v\\
&\qquad -\Big(\sum_i2\kappa_i \nabla_i v^i  +nv(H)\Big)N+h(\tn_NN, v)N+(\widetilde\Ric v)^\perp.
\end{split}
\ee

First, in the case of a hypersurface, \eqref{GaussB1} can be immediately written as
\begin{align*}
\tn^\ast \tn v&=\nabla^\ast\nabla v+s^2v-nW_{\vH}v +n[\vH, v]-\tn_{N}\tn_{N}v+\tn_{\tn_{N}N}v\\
&\quad-2\delta^{ij}h(E_i, \nabla_{E_j}v)N -\delta^{ij}(\tilde R(E_j,v)E_i)^\perp
\end{align*}
We now make the following observations.  By the Weingarten equation \eqref{Weq2}, and $\nabla^\perp_vN=0$ and $\vH=HN$ for a hypersurface, it follows
\[
-W_{\vH}v+[\vH,v]=(\tn_v\vH)^T+H[N,v]-v(H)N=H\tn_vN+H[N,v]-v(H)N=H\tn_Nv-v(H)N.
\]

 By definition of the shape operator, we have
\begin{align}\label{generalF1}
\delta^{ij}h(E_i, \nabla_{E_j}v)=\delta^{ij}g(sE_i,\nabla_{E_j}v)=\sum_{i=1}^n\kappa_i\nabla_i v^i.
\end{align}
Moreover, by symmetries of the Riemann curvature tensor and definition of the Ricci tensor
\be\label{usefullater2}
\begin{split}
\delta^{ij}(\tilde R(E_j,v)E_i)^\perp&=\delta^{ij}\tg(\tilde R(E_j,v)E_i, N)N\\
&=\delta^{ij}\tilde R(E_j, v, E_i, N)N+\tilde R(N,v,N,N)N\\
&=-(\widetilde \Ric v)^\perp.
\end{split}
\ee

Finally, \eqref{b4projectBa} follows by the Gauss formula \eqref{gauss} since for a hypersurface, $\nabla^\perp_NN=0$.

Next, for convenience, here is \eqref{b4projectB2}

\be\label{b4projectB2a}
\begin{split}
\tilde \nabla^\ast \tilde \nabla v &=\nabla^\ast  \nabla v+s^2v+nH\tn_Nv-\tilde \nabla_N\tilde\nabla_N v+\nabla_{\tilde\nabla_N N}v\\
&\qquad -\left((\dv h)(v)+2\sum_i\kappa_i \nabla_i v^i-h(\tn_NN, v)\right) N.
\end{split}
\ee

Then for a hypersurface, \eqref{GaussB2} becomes
\be\nonumber
\begin{split}
\tn^\ast \tn v &=\nabla^\ast\nabla v+s^2v +nH\tn_{N}v-\tilde \nabla_N\tilde\nabla_N v+\nabla_{\tilde\nabla_N N}v\\
&\qquad -(\tr \nabla^B \two) (v)-2\delta^{ij}h(E_i, \nabla_{E_j}v)N+h(\tn_NN,v)N.
\end{split}
\ee
We next have from \eqref{trB} and \eqref{nablaBs}
\begin{align*}
(\tr \nabla^B \two) (v)&=\delta^{ij} (\nabla^B_{E_j}\two)(E_i, v)\\
&=\delta^{ij} ({E_j}h(E_i, v)-h(\nabla_{E_j}E_i, v)-h(E_i, \nabla_{E_j}v))N\\
&=\delta^{ij}\nabla h(E_i, v, E_j)N\\
&=(\dv h)(v)N.
\end{align*}
These equations and \eqref{generalF1} imply \eqref{b4projectB2a} as needed.

\subsection{Proof of Corollary \ref{corsimple}} 
Equation \eqref{divh} follows from comparing the normal components of \eqref{b4projectB} and \eqref{b4projectB2} and $(\widetilde \Ric v)^\perp=\widetilde \Ric(v,N)N$.  We can also show this directly since using \eqref{PMC}, \eqref{usefullater2}, and the definition of a covariant derivative of a $2$-tensor, we have
\[
(\dv h)(v)=\sum_i\nabla h(E_i, v, E_i)=-\widetilde \Ric(v,N)+\sum_i \left(nv(\kappa_i)-2h(\nabla_v E_i, E_i) \right),
\]
and the last term is zero by the compatibility with the metric and using the frame consisting of the eigenvectors of $s$ at $p$.  Alternately, \eqref{divh} could be also obtained from formula \eqref{divII} for general codimensions.  To see equation \eqref{divII} we use the symmetry of $\two$ and general Codazzi equation \eqref{CodazziG} as follows.
\begin{align*}
(\tr \nabla^B\two)(v)=\delta^{ij}(\nabla_{E_j}^B\two)(E_i,v)
&=\delta^{ij}(\nabla_{E_j}^B\two)(v,E_i)\\
&=\delta^{ij}\left((\nabla_{v}^B\two)(E_j,E_i)+(\tilde R(E_j,v)E_i)^\perp\right)\\
&=\delta^{ij}\left((\nabla_{v}^B\two)(E_j,E_i)+(\tilde R(E_j,v)E_i)^\perp\right)\\
&=\delta^{ij}\left(\nabla_{v}^\perp\two(E_j,E_i)-2\two(\nabla_vE_j, E_i)+(\tilde R(E_j,v)E_i)^\perp\right),
\end{align*}
and the result follows by \eqref{surprise}.
\subsection{Proof of Theorem \ref{ThmR}}

By definition, and using $v^{n+l}(p)=0$ for $1\leq l \leq k$,  $p\in M$,
\be\label{deriveRic}
\widetilde \Ric v=\widetilde \Ric (v,\cdot)^\sharp=v^i\tilde R_{i\alpha}E_\alpha=v^i\tilde R_{ij}E_j+\sum_lv^i\tilde R_{i(n+l)}N_l.
\ee
Next
\be\label{convention}
\tilde R_{ij}=\delta^{\alpha \beta}\tilde R_{\alpha ij \beta}=\sum_l\tilde R_{(n+l) ij (n+l)}+\delta^{ab}\tilde R_{a ij b},
\ee
and from \eqref{GaussC}, \eqref{vH}, \eqref{Wm}
\begin{align*}
\delta^{ab}\tilde R_{a ij b}&=\delta^{ab}(R_{aijb}-\tilde g(\two(E_a, E_b),\two(E_i,E_j))+\tilde g(\two(E_a, E_j), \two(E_i, E_b)) )\\
&=R_{ij}-n\tg(\vH, \two(E_i, E_j))+\delta^{ab}\tilde g(\two(E_a, E_j), \two(E_i, E_b)) \\
&=R_{ij}-ng(W_{\vH}E_i,  E_j)+\delta^{ab}\tilde g(\two(E_a, E_j), \two(E_i, E_b)) .
\end{align*}

Plugging into \eqref{deriveRic} and using Lemma \ref{s2L} we have
\begin{align}
\widetilde \Ric v=\sum_lv^i \tilde R_{(n+l) ij (n+l)}E_j+\Ric v-nW_\vH v+\sum_l W_{N_l}W_{N_l}v+\sum_lv^i\tilde R_{i(n+l)}N_l.
\end{align}
Observe
\[
\sum_lv^i\tilde R_{i(n+l)}N_l=(\widetilde \Ric v)^\perp,
\]
as well as
\begin{align*}
\sum_l v^i\tilde R_{(n+l) ij (n+l)}E_j&=\sum_l \delta^{ij}\tilde R(N_l, v, E_i, N_l)E_j\nonumber\\
%&=\tilde R(N, v, E_k, N)E_k\nonumber\\
&=\sum_l\delta^{ij}\tilde R(N_l, E_i, v, N_l)E_j.
\end{align*}
This gives \eqref{genRicM} as needed.  The formula for the Euclidean submanifold, \eqref{genRicR}, follows then immediately by using that  $\tilde R$ and $\widetilde \Ric$ are identically zero on $\R^{n+k}$.  In the case of a hypersurface, \eqref{genRicM} implies \eqref{genRic} since
\[
-nW_\vH v+\sum_l W_{N_l}W_{N_l}v=n(\tn_v(HN))^T+s^2v=-nHsv+s^2v.
\]
This and \eqref{genRicR} imply \eqref{Ricci1}.  The proof is completed.

We now show that a curvature term appearing in \eqref{genRic}  is invariant under a change of ON basis.
\begin{lemma}\label{indep}
The term $\sum_{l=1}^k \delta^{ij}\tilde R(N_l, E_i, v, N_l)E_j$ is invariant under a change of adapted ON basis.
\end{lemma}
\begin{proof}
 Let $F_i$ be another ON basis for $T_pM$, and $\tilde N_l$ be another ON basis for $N_pM$.  We can write
 $$E_i=A^j_iF_j,\quad N_l=\sum_{m=1}^kB^m_l\tilde N_m$$
  where $(A^j_i)$ is an orthogonal $n\times n$ change of basis matrix, and $(B_l^m)$ is an orthogonal $k\times k$ change of basis matrix.  Then by the tensorial properties of $\tilde R$ and $(A^j_i), (B_l^m)$ being orthogonal matrices, we have
\begin{align*}
\sum_{l=1}^k\sum_{i=1}^n\tilde R(N_l, E_i, v, N_l)E_i&=\sum_{l=1}^k\sum_{i=1}^n\tilde R(B^m_l\tilde N_m, A^j_iF_j ,v, B^{m'}_l\tilde N_{m'} )A^{j'}_iF_{j'}\\
&=\sum_{l=1}^k\sum_{i=1}^n\tilde R(\tilde N_m, F_j ,v, \tilde N_{m'} )B^m_lB^{m'}_l A^j_iA^{j'}_iF_{j'}\\
&=\sum_{m=1}^k\sum_{j=1}^n\tilde R(\tilde N_m, F_j ,v, \tilde N_{m} )F_{j},
\end{align*}
as needed.
\end{proof}

\subsection{Proof of Corollary \ref{cBE}}
The formulas \eqref{GaussB1R} and \eqref{b4projectRBe} follow directly from \eqref{GaussB1} and \eqref{b4projectB}, respectively, since the ambient curvature terms in this case are zero.  The remaining formulas follow from \eqref{genRicR} and \eqref{Ricci1}.

\subsection{Proof of Corollary \ref{corDef}}
Here, for a general submanifold of codimension $k\geq1$, the formulas follow directly from using that 
\[
\tilde Lv=-2\widetilde \dv\widetilde \Def v=\tn^\ast \tn v-\widetilde \Ric v+\tilde\nabla \widetilde\dv,
\]
and similarly for $Lv$, and using the previously derived formulas for the Bochner Laplacian, $\widetilde \Ric$, as well as Lemma \ref{lemdiv}.

For a hypersurface, we first put together \eqref{b4projectB}, \eqref{genRic} and \eqref{graddiv}.  This gives
\be\label{org1}
\begin{split}
\tilde Lv &=Lv+nHsv+nH\tn_Nv-\tilde \nabla_N\tilde\nabla_N v+\nabla_{\tilde\nabla_N N}v+\mathcal E_2(v)\\
&\qquad -\left(2\sum_i\kappa_i \nabla_i v^i  +nv(H)-h(v, \tn_NN)\right)N+\nabla^\perp\widetilde \dv v,
\end{split}
\ee
with 
\[
\mathcal E_2(v)=-\delta^{ij}\tilde R(N, E_i, v, N)E_j+\nabla \dv^\perp v.
\]
 We now examine $\nabla^\perp \widetilde \dv v$.  For codimension one, 
 \[
 \nabla^\perp \widetilde \dv v=N \widetilde \dv v N, 
 \]
 so from \eqref{diveq} and compatibility with the metric we get
\be\label{above1}
\nabla^\perp \widetilde \dv v=N\tg(\tn_{E_\al} v, E_\al)N=\left(N(\tn_i v^i)+\tg(\tn_N\tn_N v, N)+\tg(\tn_N v, \tn_NN)\right)N.
\ee
Since for codimension one, the normal component of $\tn_N\tn_N v$ is just $\tg(\tn_N\tn_N v, N)$, we have a cancelation of the normal component of $\tn_N\tn_N v$ in \eqref{org1} with the middle term on the right hand side of \eqref{above1}, which gives
\begin{align}
\tilde Lv &=Lv+nHsv+nH\tn_Nv-(\tilde \nabla_N\tilde\nabla_N v)^T+\nabla_{\tilde\nabla_N N}v+\mathcal E_2(v)\nonumber\\
& -\Big(2\sum_i\kappa_i \nabla_i v^i  +nv(H)-h(v, \tn_NN)\Big)N+\Big(N(\tn_i v^i)+\tg(\tn_N v, \tn_NN)\Big)N.\label{corlastline}
\end{align}
Finally, to arrive at $\mathcal N_3(v)$ in \eqref{b4projectBdef}, recalling $\tn_NN$ is tangential, consider
\begin{align*}
h(v, \tn_NN)+\tg(\tn_N v, \tn_NN)&=g(sv, \tn_NN)+\tg(\tn_N v, \tn_NN)\\
&=-g(\tn_vN, \tn_NN)+\tg(\tn_N v, \tn_NN)\\
&=\tg([N,v], \tn_NN),
\end{align*}
where we used \eqref{Weq} in the second line.  Inserting this into \eqref{corlastline} gives \eqref{b4projectBdef} with the first line for $\mathcal N_3$.  The second line for $\mathcal N_3$ follows from Corollary \ref{corsimple}.

\subsection{Proof of Corollary \ref{corDefdiv}}
For a hypersurface, we begin with \eqref{org1} and cancel the terms involving divergence of $v$.  This gives
\be\label{usefulnow}
\begin{split}
\tn^\ast\tn v-\widetilde \Ric v &=\nabla^\ast\nabla v-\Ric v+nHsv+nH\tn_Nv\\
&\quad-\tilde \nabla_N\tilde\nabla_N v+\nabla_{\tilde\nabla_N N}v-\delta^{ij}\tilde R(N, E_i, v, N)E_j\\
&\quad -\Big(2\sum_i\kappa_i \nabla_i v^i  +nv(H)-h(v,\tn_NN)\Big)N.
\end{split}
\ee
From \eqref{diveq}, \eqref{divperp} and Lemma \ref{lemdiv} we have
\[
0=\tg(\tn_N v, N). 
\]
So
\[
\tn_Nv=(\tn_Nv)^T,
\]
and from \eqref{Weq}
\[
sv=-(\tn_vN)^T.
\]
Putting these together we obtain
\begin{align*}
\tn^\ast\tn v-\widetilde \Ric v &=\nabla^\ast\nabla v-\Ric v+nH[N,v]^T-\tilde \nabla_N\tilde\nabla_N v+\nabla_{\tilde\nabla_N N}v\\
&\qquad -\left(2\sum_i\kappa_i \nabla_i v^i  +nv(H)-h(v,\tn_NN)\right)N-\delta^{ij}\tilde R(N, E_i, v, N)E_j,
\end{align*}
which using Corollary \ref{corsimple} is seen to be \eqref{b4projectBdefdiv}.  The formula for the Euclidean hypersurface follows then immediately.
\subsection{Proof of Corollary \ref{corHodgeLap}}
All formulas follow directly from using the formulas for the Bochner Laplacian in Theorem \ref{gaussThmM} and for the Ricci operator in Theorem \ref{ThmR}, and the fact that by the Bochner-Weitzenb\"ock formula \eqref{BW}, the Hodge Laplacian can be written as the Bochner Laplacian plus the Ricci operator.
 
\subsection{Proof of Corollary \ref{corProject}}
Considering the right hand side of formulas \eqref{GaussB1} and \eqref{GaussB2}, the first line in both formulas stays the same as each term is already tangential, we drop the third line of each formula, and take the tangential part of the second lines to arrive at \eqref{GaussB1p} since
\be\label{usesoon}
-nW_{\vH}v+n[\vH,v]^T=n(\tn_{\vH}v)^T.
\ee
The formula \eqref{projectB} for the hypersurface can then be obtained by either writing \eqref{GaussB1p} when codimension $k=1$ (and using the Gauss formula \ref{gauss} and $\tn_N^\perp N=0$) or projecting \eqref{b4projectB} or \eqref{b4projectB2}.

To see the formula \eqref{DefprojectM}, the steps are similar: drop $\mathcal N_1$ in \eqref{GaussDef} or $\mathcal N_2$ in \eqref{GaussDef2}, use \eqref{usesoon}, and project $-\sum_{l=1}^k\tn_{N_l}\tn_{N_l}v+\sum_{l=1}^k\tn_{\tn_{N_l}N_l}v$.

Next, the formula \eqref{projectBdef2} follows from \eqref{b4projectBdef} by dropping $\mathcal N_3$, and projecting $[N,v]$.  

The proofs of the remaining formulas are analogous.
 \subsection{Proof of Corollary \ref{corProjectE}}
 There are many ways to see the formulas.  One is to note that the formula \eqref{GaussB1p} remains unchanged when the ambient manifolds is $\R^{n+k}$, and then apply the formula \eqref{genRicR} for $\Ric v$ to arrive at \eqref{projectBR1}.  From there we can apply \eqref{BW} and \eqref{genRicR} again together with \eqref{Weq2} to obtain \eqref{Hodgepk}.
 
 The formula \eqref{Defproject} follows immediately from \eqref{DefprojectM}, where we also explicitly write out the operators $\tilde Lv$ and $Lv$.
 
 In the case of a hypersurface, the proof is exactly analogous, just starting with the formula 
\eqref{projectB}, which is unchanged, and then continuing with \eqref{Ricci1}, and so on.  The new component is the formula in the case of a divergence free vector field \eqref{projectBdefdivfreeR2}, which follows from \eqref{projectBR}-\eqref{projectBdefdiv} and $[N,v]^T=[N,v]=\tn_Nv+sv$ for divergence free vector fields.
%%%%%%%%%%%%%%%%%%%%%%%%%%%%%%%%%%%%%%%%%%%%%%%%%%%%%%%%%%%%%%%%%%%%%%%%%%%%%%%%%%%%%%%%%%%%
%%%%
%%%%.    					  Proof of the Lie Derivative formula for surfaces of revolution
%%%%
%%%%%%%%%%%%%%%%%%%%%%%%%%%%%%%%%%%%%%%%%%%%%%%%%%%%%%%%%%%%%%%%%%%%%%%%%%%%%%%%%%%%%%%%%%%%

\section{Proof of Theorem \ref{thm2}}
For concreteness, we work with formula \eqref{projectRBe2} in Corollary \ref{corProjectE}.  The goal is to connect  \eqref{projectRBe2} that we now know holds to formula \eqref{formulaSR}.  To that end, we begin with a proposition.

 \begin{prop}\label{doublelie} With the same setting as in Theorem \ref{thm2}, we have
\be\nonumber
\begin{split}
(\mathfrak{\iota}_{\mathrm{S}}^*\mathcal L_{N}\mathcal L_{N}v^\flat)^{\sharp_S}&=(\tn_N\tn_Nv)^{T_S}- 2 s(\mathfrak{\iota}_{\mathrm{S}}^*\mathcal L_{N}v^\flat)^{\sharp_S}- s^2v+g(v, \tn_N\tn_{E_i}N)E_i\\
  &\qquad+g(\tn_NN, \tn_NN)g(v,E_1)E_1\\
  &=(\tn_N\tn_Nv)^{T_S}-2 s(\mathfrak{\iota}_{\mathrm{S}}^*\mathcal L_{N}v^\flat)^{\sharp_S}- s^2v+v^iN(c^i_{i3})E_i
  +(c^3_{13})^2v^1E_1.
\end{split}
\ee
\end{prop}

\begin{proof}
By Corollary \ref{useful} and properties of the connection we have
\begin{align*}
\ip{\mathcal L_{N}\mathcal L_{N}v^\flat, E_i}=g(\tn_{N}(\mathcal L_{N}v^\flat)^\sharp, E_i)+g((\mathcal L_{N}v^\flat)^\sharp, \tn_{E_i}N).
\end{align*}
To handle the first term we note
\[
\tn_{N}(\mathcal L_{N}v^\flat)^\sharp=(\tn_{N}\mathcal L_{N}v^\flat)^\sharp,
\]
and that again by Corollary \ref{useful} and properties of the connection
\begin{align*}
\tn_{N}\mathcal L_{N}v^\flat&=N(\mathcal L_{N}v^\flat)_\alpha E^\alpha+(\mathcal L_{N}v^\flat)_\alpha\tn_{N}E^\alpha\\
&=g(\tn_N\tn_{N}v,E_\alpha)E^\alpha+g(\tn_{N}v, \tn_NE_\alpha)E^\alpha+g(\tn_N v, \tn_{E_\alpha}N)E^\alpha\\
&\quad+g(v, \tn_N\tn_{E_\alpha}N)E^\alpha+(\mathcal L_{N}v^\flat)_\alpha\tn_{N}E^\alpha.
\end{align*}

Using \eqref{tnN} we obtain
\begin{align*}
\ip{\mathcal L_{N}\mathcal L_{N}v^\flat, E_i}&= 
g(\tn_N\tn_{N}v,E_i)+g(\tn_{N}v, \tn_NE_i)+g(\tn_N v, \tn_{E_i}N)+g(v, \tn_N\tn_{E_i}N)\\
&\quad+\delta_{i1}{c^3_{13}}(\mathcal L_{N}v^\flat)_3+g((\mathcal L_{N}v^\flat)^\sharp, \tn_{E_i}N).
 \end{align*}

From 
\[
\tn_{E_i}N=-\kappa_i E_i,
\]
we get

\begin{align*}
\ip{\mathcal L_{N}\mathcal L_{N}v^\flat, E_i}&= 
g(\tn_N\tn_{N}v,E_i)+g(\tn_{N}v, \tn_NE_i)-\kappa_ig(\tn_N v, E_i)+g(v, \tn_N\tn_{E_i}N)\\
&\quad+\delta_{i1}{c^3_{13}}(\mathcal L_{N}v^\flat)_3-\kappa_i (\mathcal L_{N}v^\flat)_i,
 \end{align*}
which we can further simplify using \eqref{tnN} and \eqref{useful_div} as 
\begin{align*}
\ip{\mathcal L_{N}\mathcal L_{N}v^\flat, E_i}&= 
g(\tn_N\tn_{N}v,E_i)-\kappa_ig(\tn_N v, E_i)+g(v, \tn_N\tn_{E_i}N)\\
&\quad+\delta_{i1}{c^3_{13}}(\mathcal L_{N}v^\flat)_3-\kappa_i (\mathcal L_{N}v^\flat)_i,
 \end{align*}
and by \eqref{Lie3} this is 
\begin{align*}
\ip{\mathcal L_{N}\mathcal L_{N}v^\flat, E_i}&= 
g(\tn_N\tn_{N}v,E_i)-\kappa_ig(\tn_N v, E_i)+g(v, \tn_N\tn_{E_i}N)\\
&\quad+\delta_{i1}({c^3_{13}})^2v^1-\kappa_i (\mathcal L_{N}v^\flat)_i.
 \end{align*}

We now again use Corollary \ref{useful}
\[
g(\tn_N v, E_i)=(\mathcal L_N v^\flat)_i-g(v, \tn_{E_i}N)=(\mathcal L_N v^\flat)_i+\kappa_i v^i,
\]
which is
\[
(\tn_N v)^T=(\mathcal L_N v^\flat)^T+sv,
\]
to arrive at
\begin{align*}
\ip{\mathcal L_{N}\mathcal L_{N}v^\flat, E_i}&= 
g(\tn_N\tn_{N}v,E_i)-(\kappa_i)^2 v^i+g(v, \tn_N\tn_{E_i}N)\\
&\quad+\delta_{i1}({c^3_{13}})^2v^1-2\kappa_i (\mathcal L_{N}v^\flat)_i.
 \end{align*}
Using \eqref{tnN} again, we can write this also as

\begin{align*}
\ip{\mathcal L_{N}\mathcal L_{N}v^\flat, E_i}&= 
g(\tn_N\tn_{N}v,E_i)-(\kappa_i)^2 v^i+g(v, \tn_N\tn_{E_i}N)\\
&\quad+\delta_{i1}g(\tn_NN, \tn_NN)g(v,E_1)-2\kappa_i (\mathcal L_{N}v^\flat)_i,
 \end{align*}
 which is equivalent to
 \be\nonumber
\begin{split}
(\mathfrak{\iota}_{\mathrm{S}}^*\mathcal L_{N}(\mathcal L_{N}v^\flat))^{\sharp_S}&=(\tn_N\tn_Nv)^{T_S}- 2 s(\mathfrak{\iota}_{\mathrm{S}}^*\mathcal L_{N}v^\flat)^{\sharp_S}- s^2v+g(v, \tn_N\tn_{E_i}N)E_i\\
&\quad  +g(\tn_NN, \tn_NN)g(v,E_1)E_1.
\end{split}
\ee

By \eqref{Weq1}
\[
 \tn_N\tn_{E_i}N=\tn_N(c^i_{i3}E_i)=N(c^i_{i3})E_i+c^i_{i3}\tn_N E_i,
\]
which by \eqref{tnN} and \eqref{v0} will only contribute $N(c^i_{i3})E_i$, so for computational purposes we can also write
\be
\begin{split}
(\mathfrak{\iota}_{\mathrm{S}}^*\mathcal L_{N}(\mathcal L_{N}v^\flat))^{\sharp_S}&=(\tn_N\tn_Nv)^{T_S}-2 s(\mathfrak{\iota}_{\mathrm{S}}^*\mathcal L_{N}v^\flat)^{\sharp_S}- s^2v+v^iN(c^i_{i3})E_i
  +(c^3_{13})^2v^1E_1,
\end{split}
\ee
and the proposition is proved.
\end{proof}
We begin now comparing formulas \eqref{formulaSR} and \eqref{projectRBe2}.  Copying over for convenience, from \eqref{projectRBe2} we have
\be\label{project2}
(\tilde \nabla^\ast \tilde \nabla v)^T=-\Delta_hv-2\Ric v+2H[N,v]^T-(\tilde \nabla_N\tilde\nabla_N v)^T+(\tilde\nabla_{\tilde\nabla_N N}v)^T,
\ee
where we used the Bochner-Weitzenb\"ock formula \eqref{BW}.  From \eqref{formulaSR}
\begin{equation}\label{formulaSR2a}
\begin{split}
\iota_{S}^* \Big \{ -\triangle v \Big \} & =   -\triangle_{h}\left( \iota_{S}^* v\right) -  \iota_{S}^* \left\{\mathcal{L}_N\mathcal{L}_N v\right\}  + \iota_{S}^* \left\{ ( \kappa_1-\kappa_2)\mathcal{L}_{N} v\right\} 
+ \frac 1{\ag^2} \iota_{S}^* \Big \{ \mathcal{L}_{\mathrm{Y}} \big ( v \big ) \Big\}\\
% \sqrt{K_{S}}
&\qquad+ 2(\kappa_2-\kappa_1)(\mathcal L_{N} v)_1 E^1-2\left(\frac{E_1(\ag)}{\ag}\right)^2v_1E^1,
\end{split}
\end{equation}
where
\[
Y=\frac{E_1(\abs{\nabla\rho})}{f}E_1.
\]
Using Proposition \ref{doublelie}, this means we need
\begin{align*}
&2 s(\mathfrak{\iota}_{\mathrm{S}}^*\mathcal L_{N}v^\flat)^{\sharp_S}+s^2v-v^iN(c^i_{i3})E_i
  -3(c^3_{13})^2v^1E_1\\
&\quad+ \Big( {\kappa_1-\kappa_2}\Big)(\iota_{\mathrm{S}}^*\mathcal{L}_{N} v^\flat)^\sharp
+f^2 (\iota_{\mathrm{S}}^*  \mathcal{L}_{\mathrm{Y}}v^\flat)^\sharp + 2 (\kappa_2-\kappa_1)(\iota_{\mathrm{S}}^* \mathcal L_{N} v)_1 E_1\\
&=-2\Ric v+2H[N,v]^T +(\tilde\nabla_{\tilde\nabla_N N}v)^T.
\end{align*}
From \eqref{tnN}-\eqref{t12} and \eqref{v0}, we have
\[
(\tilde\nabla_{\tilde\nabla_N N}v)^T=c^3_{13}E_1(v^i)E_i,
\]
so from \eqref{Liey} we see we need to have
\begin{align*}
&2 s(\mathfrak{\iota}_{\mathrm{S}}^*\mathcal L_{N}v^\flat)^{\sharp_S}+s^2v -(N(c^i_{i3})+\delta_{i1}{3(c^3_{13})^2})v^iE_i +\Big( {\kappa_1-\kappa_2}\Big)(\iota_{\mathrm{S}}^*\mathcal{L}_{N} v^\flat)^\sharp\\
&\quad 
+f^2\delta_{i1}E_i(c^3_{13}/f^2)v^1E_1+\delta_{i2}c^3_{13}c^2_{21}v^2E_2 +  2 (\kappa_2-\kappa_1)(\iota_{\mathrm{S}}^* \mathcal L_{N} v)_1 E_1\\
&=-2\Ric v+2H[N,v]^T.
\end{align*}
Next from Lemma \ref{Lies}, it follows we need to show
\begin{align*}
&2 s(\mathfrak{\iota}_{\mathrm{S}}^*\mathcal L_{N}v^\flat)^{\sharp_S}+s^2v -(N(c^i_{i3})+\delta_{i1}{3(c^3_{13})^2})v^iE_i +\Big( {\kappa_1-\kappa_2}\Big)(\iota_{\mathrm{S}}^*\mathcal{L}_{N} v^\flat)^\sharp\\
&\quad 
+f^2\delta_{i1}E_i(c^3_{13}/f^2)v^1E_1+\delta_{i2}c^3_{13}c^2_{21}v^2E_2 +  2 (\kappa_2-\kappa_1)(\iota_{\mathrm{S}}^* \mathcal L_{N} v)_1 E_1\\
&=-2\Ric v+2H((\mathcal L_{N}v^\flat)^\sharp)^T+4Hsv.
\end{align*}

From this, we derive the need for the following two equations to hold, one corresponding to the coefficient of the Lie derivative term, and one to $v^i$.
\begin{align}
2 {\kappa_i} -\kappa_2+2\delta_{i1} (\kappa_2-\kappa_1)
&=\kappa_2,\label{main1}\\
-(N(c^i_{i3})+\delta_{i1}{3(c^3_{13})^2})+f^2 (\delta_{i1}E_i(\frac{c^3_{13}}{f^2})+\frac1{f^2}\delta_{i2}c^3_{13}c^2_{21})&=(\kappa_i)^2\label{main2}.
\end{align}

We see \eqref{main1} holds as needed, both for $i=1$ and $i=2$.  We next discuss \eqref{main2}.
\subsection{Verifying equation \ref{main2} }
In computations below, we always evaluate at $\rho=1$.  
We begin with $i=2$ as it is simpler.  For $i=2$, \eqref{main2} is
\be
-N(c^2_{23})+ c^3_{13}c^2_{21}=(\kappa_2)^2\label{main22}.
\ee
A direct computation using \eqref{grr}, \eqref{c223}, and \eqref{kappas} shows
\begin{align}
N(c^2_{23})&= \frac{\kappa_2}{f}+\frac{(\dot a a+\dot b b)}{f}\left (\frac{\ddot b}{a}-\frac{\dot a \dot b}{a^2}\right).
\end{align}
Similarly, from \eqref{c313} and \eqref{c212}
\be
c^3_{13}c^2_{21}=\left(\frac{a\ddot b-b\ddot a}{f}\right)\frac {\dot a} a.
\ee
Plugging into \eqref{main22}, and using that we have a unit speed parametrization, we can verify \eqref{main22} holds as needed.

For $i=1$, \eqref{main2} becomes
\be\label{main21}
-N(c^1_{13})-3(c^3_{13})^2+E_1(c^3_{13})-2\frac{c^3_{13}}{f}E_1(f)=(\kappa_1)^2.
\ee
From \eqref{c313} and $\ag=f^{-1}$, we have
\[
c^3_{13}=-\frac{E_1(f)}{f}=-\frac{\dot f}{f},
\]
and (recalling we are evaluating at $\rho=1$)
\[
E_1(c^3_{13})=-\partial_t (\frac{\dot f}{f})=-\frac{\ddot f}{f}+\Big(\frac {\dot f}{f}\Big)^2,
\]
so we can simplify \eqref{main21} as follows
\begin{align}\label{main21a}
 -N(c^1_{13})-\frac{\ddot f}{f}=(\kappa_1)^2.
\end{align}
Next, from  \eqref{c13} and \eqref{kappas}
\[
c^1_{13}=-\frac{\ddot a a+\ddot bb}{\rho f}, \quad \kappa_1=\frac{\ddot a a+\ddot bb}{f},
\]

and using $g^{\rho t}=-\frac{g_{\rho t}}{f^2}$, a direct computation gives
 
\begin{align*}
N(c^1_{13})&=\frac{\kappa_1}{f}-fg^{\rho t}\partial_t\left(\frac{\ddot a a+\ddot b b}{f}\right)\\
&=\frac{\kappa_1}{f}(1-\frac{ g_{\rho t}}{f}\dot f)+\frac{g_{\rho t}}{f^2}\partial_t(\ddot a a+\ddot b b).
\end{align*}
 
 We now make the following observations.  From the unit speed parametrization, we can write
 \be\label{kappa1new}
\kappa_1=\frac{\dot g_{\rho t}-1}{f},
\ee
and a computation shows
\be\label{lesspainful}
f\dot f= g_{\rho t}(1-\dot  g_{\rho t}).
\ee
It follows
\[
-g_{\rho t}\kappa_1=\dot f.
\]
Using this, we have
\begin{align*}
N(c^1_{13}) =\frac{\kappa_1}{f}+\frac{\dot f^2}{f^2}+\frac{g_{\rho t}}{f^2}\ddot g_{\rho t}.
\end{align*}

Plugging into \eqref{main21a}, we see we need to verify
\be\label{interim}
\frac{\kappa_1}{f}+\frac{\dot f^2}{f^2}+\frac{g_{\rho t}}{f^2}\ddot g_{\rho t}+\frac{\ddot f}{f}=-(\kappa_1)^2.
\ee
Differentiating \eqref{lesspainful} gives
\[
g_{\rho t} \ddot g_{\rho t}+f\ddot f=\dot g_{\rho t}-(\dot g_{\rho t}^2+\dot f^2).
\]
Hence \eqref{interim} simplifies to
\be\label{interim2}
\frac{\kappa_1}{f}+\frac{\dot g_{\rho t}}{f^2}-\frac{\dot g_{\rho t}^2}{f^2}=-(\kappa_1)^2.
\ee
We can see \eqref{interim2} holds as needed by using \eqref{kappa1new}.

\subsection{Proof of Corollary \ref{corthm2}}
Formula \eqref{formulaSR2} follows from  formula \eqref{formulaSR} by using

\[
(\mathcal L_N v)_i=\frac 1 \ag (\mathcal L_{\nabla \rho} v)_i,
\]
and
\[
(\mathcal L_N \mathcal L_N v)_i=(\mathcal L_{\nabla \rho}\frac 1{\ag^2}\mathcal L_{\nabla \rho}v)_i+\frac{N(\ag)}{\ag}(\mathcal L_N v)_i-2\delta_{i1}\left(\frac{E_1(\ag)}{\ag}\right)^2v^1.
\]

 \section{Examples}\label{examples}
 In this section we consider some examples.  We begin with the sphere in $\R^3$.
 
\subsection{Sphere, $S^2\hookrightarrow \R^3$}\label{sphereEx}
We set up all the computational tools.  First, if we let $N$ be the outward pointing normal to $S^2$, then $\kappa_i=-1, i=1, 2$, so
 \[
 sv=-v \quad\mbox{and}\quad s^2v=v.
\]
 We can let $N=\partial_\rho$ , where $\rho$ is the radial variable in the spherical coordinates.  If $(\phi, \theta)$ are the standard coordinates on the sphere (without the North and South pole), denoting the polar and the azimuthal angle, respectively, then 
\[
E_1=\partial_\phi, \quad E_2=\frac{\partial_\theta}{\sin \phi},
\]
form an ON frame on the sphere, and correspond to the eigenvectors of the shape operator.  This frame also naturally extends to all of $\R^3\setminus \{0\}$ (giving the standard ON frame in spherical coordinates)
\[
E_1=\frac{\partial_\phi}{\rho}, \quad E_2=\frac{\partial_\theta}{\rho\sin \phi},\quad E_3=N=\partial_\rho.
\]
Let $v \in \mathfrak X(S^2)$, then on $S^2$, $v=v^iE_i$, and in the neighborhood of $S^2$, an extension of $v$, still denoted by $v$, satisfies
\[
v=v^\alpha E_\alpha, \quad v^3|_{S^2}=0.
\]
Then we can compute, either directly or using the formulas in Section \ref{dictionary} with $t=\phi, (a(\phi), b(\phi))=(\sin \phi, \cos \phi)$ as the parametrization,
  \be\label{tnnns}
\tn_NN=\tn_{\partial_{\rho}}\partial_\rho=0,
\ee
and 
\[
\tn_N v=\partial_\rho v^\alpha E_\alpha.
\]
So
\[
\tn_N\tn_N v=\tn_N(\partial_\rho v^\alpha E_\alpha)=\partial_\rho^2 v^\alpha E_\alpha.
\]
Also
\begin{align*}
2\kappa_i \nabla_i v^i&=-2\nabla_i v^i=-2\dv_{S^2}v,\\
(\dv h)(v)&=nv(H)=0=h(v,\tn_N N),\\
[N,v]&=\tn_N v-\tn_v N=\partial_\rho v^\alpha E_\alpha+sv,\\
\Ric v&=v,
\end{align*}
where we use Corollary \ref{corsimple} in the second line in the first equality (and it can be checked directly, too).
Then the formulas \eqref{b4projectB2} and \eqref{b4projectRBe} are the same and read
\be\label{sphere1}
\tilde \nabla^\ast \tilde \nabla v=\nabla^\ast  \nabla v+s^2v+nH\tn_Nv-\tilde \nabla_N\tilde\nabla_N v+2\dv_{S^2}v N,
\ee
whereas \eqref{b4projectRBe2} and \eqref{b4projectRB2} give
\be\label{sphere2}
\begin{split}
\tilde \nabla^\ast \tilde \nabla v &=\nabla^\ast  \nabla v-\Ric v+nH[N,v]-\tilde \nabla_N\tilde\nabla_N v+2\dv_{S^2}v N.
\end{split}
\ee

Comparing, we see that in the case of the sphere, the two formulas differ only by the second and third terms on the right.  As mentioned before, in general, the formula  \eqref{b4projectB2}, which is now \eqref{sphere1} here, is interesting, because it is the same regardless if the ambient manifold is a general manifold or the Euclidean space.  On the other hand, it does have more terms that individually are not intrinsic: the only term that is purely intrinsic is the Bochner Laplacian.  In the case of \eqref{sphere2}, the intrinsic terms are what is equivalent to the deformation Laplacian (for divergence free vector fields, cf, \eqref{BW}): $\nabla^\ast  \nabla v-\Ric v$.

If, as discussed in Section \ref{closerlook}, from the tangential terms we took the intrinsic and extrinsic terms, but not those that depend on the extension, we would get both from \eqref{sphere1} and \eqref{sphere2} (using \eqref{Ricci1} and \eqref{Weq})
\be
(\tilde \nabla^\ast \tilde \nabla v)^T=\nabla^\ast  \nabla v+s^2v=\nabla^\ast  \nabla v+v=\Delta_h v,
\ee
by \eqref{BW}, which is the Hodge Laplacian.

We continue with this example to illustrate what happens when we consider different extensions.  From the above computations we have
\begin{align}\label{usenext}
\tilde \nabla^\ast \tilde \nabla v =\nabla^\ast  \nabla v+v-2\partial_\rho v^\alpha E_\alpha-\partial_\rho^2 v^\alpha E_\alpha+2\dv_{S^2}v N.
\end{align}

We now take a closer look by considering different vector field extensions.  For example, we can extend $v=v^iE_i$, by doing \emph{nothing}, in the sense that we just let the extension be
\be\label{parallelt}
v=v^1E_1+v^2E_2,
\ee
where we think of $E_i$ as vector fields defined in all of $\R^3\setminus \{0\}$.  Then $v^i$ are independent of $\rho$ and we immediately get from \eqref{usenext}
\begin{align*}
(\tilde \nabla^\ast \tilde \nabla v)^T&=\nabla^\ast  \nabla v+ v=-\Delta_h v,\\
(\tilde \nabla^\ast \tilde \nabla v)^\perp&=2\dv_{S^2}v N.
\end{align*}
 On the other hand, if we write $v$ in coordinates, i.e., 
 \[
 v=v^\phi\partial_\phi+v^\theta \partial_\theta,
 \]
and then do an extension by doing \emph{nothing}, and write
 \[
 v=v^\phi\partial_\phi+v^\theta \partial_\theta,
 \]
then in the frame,
\[
v=\rho v^\phi E_1+\rho\sin\phi v^\theta E_2.
\]
It follows
\[
\partial_\rho v^i|_{p\in S^2}=v^i, \quad \partial^2_\rho v^i|_{p\in S^2}=0, 
\]
and $v^3=0$, so
\[
 \partial_\rho v^3=\partial^2_\rho v^3=0.
\]
 Hence
 \begin{align*}
(\tilde \nabla^\ast \tilde \nabla v)^T&=\nabla^\ast  \nabla v-v,\\
(\tilde \nabla^\ast \tilde \nabla v)^\perp&=2\dv_{S^2}v N,
\end{align*}
 where the tangential component is the deformation Laplacian.  
 
 As mentioned in the introduction, these two types of extensions were discussed in \cite{Orszag, Yamada_2018}, and the second one was used in \cite{CCM17}.   
 
 Therefore, we see that even in the case of the sphere, the formula for the projected operator depends on the extension used.  We will look at this further in an upcoming article.

  \subsection{Minimal surfaces}
  In the case of a minimal surface, say catenoid embedded in $\R^3$, we have $H=0$.  Then from \eqref{Ricci1}, we have $\Ric v=-s^2v$, and from Corollary \ref{corsimple} $\dv h=0$, so all the formulas \eqref{b4projectB2}, \eqref{b4projectRBe}, \eqref{b4projectRBe2} and \eqref{b4projectRB2} simplify to
 
 \be
\begin{split}
(\tilde \nabla^\ast \tilde \nabla v)^T &=\nabla^\ast  \nabla v-\Ric v-(\tilde \nabla_N\tilde\nabla_N v)^T+\nabla_{\tilde\nabla_N N}v,
\end{split}
\ee
and
\be 
\begin{split}
(\tilde \nabla^\ast \tilde \nabla v)^\perp &=-(\tilde \nabla_N\tilde\nabla_N v)^\perp-\left(2\sum_i\kappa_i \nabla_i v^i -h(v,\tn_NN)\right)N.
\end{split}
\ee
 
\subsection{Sphere, $S^2\hookrightarrow S^3$} Further simplifications can be obtained for totally geodesic submanifolds.  For example, since $S^2$ is a totally geodesic submanifold of $S^3$, the second fundamental form is identically zero, which in turn implies that the principal curvatures are zero.  Also, since in general on $S^n$, $\Ric v=(n-1)v$, we must have $(\widetilde\Ric v)^\perp=0$ on $S^2$.  Then the two Gauss formulas for the Bochner Laplacian, \eqref{b4projectB}  and \eqref{b4projectB2} reduce to
 
\be\label{b4projectBS2S3}
\begin{split}
\tilde \nabla^\ast \tilde \nabla v &=\nabla^\ast  \nabla v -\tilde \nabla_N\tilde\nabla_N v+\nabla_{\tilde\nabla_N N}v.
\end{split}
\ee

In addition, the deformation Laplacian formula \eqref{b4projectBdef} gives
\be\label{sphereS3}
\begin{split}
\tilde L v&=Lv-(\tilde \nabla_N\tilde\nabla_N v)^T+\nabla_{\tilde\nabla_N N}v+\mathcal E_2(v)+\mathcal N_3(v),
\end{split}
\ee
where
\begin{align*}
\mathcal E_2(v)&=-\delta^{ij}\tilde R(N, E_i, v, N)E_j+\nabla\dv^\perp v,\\
\mathcal N_3(v)&=\Big(N (\tn_i v^i)+\tg([N,v], \tn_NN)\Big)N.
\end{align*}
If the vector field is divergence free on $S^2$ and on $S^3$, then from the formula \eqref{b4projectBdefdiv} we get
 \be\label{sphereS4}
\begin{split}
\tn^\ast\tn v-\widetilde\Ric v&=\nabla^\ast\nabla v-\Ric v-\tilde \nabla_N\tilde\nabla_N v+\nabla_{\tilde\nabla_N N}v-\delta^{ij}\tilde R(N, E_i, v, N)E_j.
\end{split}
\ee
Finally, from the Hodge Laplacian formula \eqref{HodgeLapH} we obtain
 \be
\begin{split}
-\tilde \Delta_hv &= -\Delta_hv -\tilde \nabla_N\tilde\nabla_N v+\nabla_{\tilde\nabla_N N}v +\delta^{ij}\tilde R(N, E_i, v, N)E_j.
\end{split}
\ee
 We note that in fact, all the above formulas are independent of the dimension and would apply to any codimension one space form $M$ embedded as a totally geodesic submanifold into space form $\tilde M$ with the same sectional curvature as $M$.
 \subsection{Hyperbolic space: $\mathbb H^n\hookrightarrow \R^{n+k}$}
 This is where we benefit from having formulas for general codimension $k\geq 1$.  Using Nash's embedding we can write down the Gauss formula for a hyperbolic space of any dimension $n\geq 2$.  Gathering the formulas \eqref{GaussB1pagain}-\eqref{Hodgepk} from Corollary \ref{corProjectE} we have

\begin{align*}
(\tilde \nabla^\ast \tilde \nabla v)^T &=\nabla^\ast\nabla v  +\sum_{l=1}^k W_{N_l}W_{N_l}v+n(\tn_{\vH}v)^T-\sum_{l=1}^k(\tn_{N_l}\tn_{N_l}v- \tn_{\tn_{N_l}N_l}v)^T\\
                                                       &=\nabla^\ast\nabla v  -\Ric v +n[\vH, v]^T-\sum_{l=1}^k(\tn_{N_l}\tn_{N_l}v- \tn_{\tn_{N_l}N_l}v)^T\\
                                                       &=- \Delta_h v +2\sum_{l=1}^k W_{N_l}W_{N_l}v+n (\tn_{\vH}v)^T-nW_{\vH}v 
-\sum_{l=1}^k(\tn_{N_l}\tn_{N_l}v- \tn_{\tn_{N_l}N_l}v)^T.
\end{align*} 
 
 To see the normal component, we use Corollary \ref{cBE} (or Corollary \ref{corHodgeLap}).  This gives
 
 \begin{align*}
(\tilde \nabla^\ast \tilde \nabla v)^\perp &=n[\vH, v]^\perp-\sum_{l=1}^k(\tn_{N_l}\tn_{N_l}v-\tn_{\tn_{N_l}N_l}v)^\perp-2\delta^{ij}\two(E_i, \nabla_{E_j}v)\\
&=n (\tn_{\vH}v)^\perp-\sum_{l=1}^k(\tn_{N_l}\tn_{N_l}v- \tn_{\tn_{N_l}N_l}v)^\perp -(\tr \nabla^B \two) (v)-2\delta^{ij}\two(E_i, \nabla_{E_j}v).
\end{align*}
As before, the actual formulas depend on the behavior of $v$ in the neighborhood of $p \in \mathbb H^n$.  For example, if $v$ was parallel in the direction of the mean curvature vector field $\vH$, then the terms involving $\tn_{\vH}v$ could be discarded.

\subsection{Theorem \ref{thm2}: sphere}
For convenience we copy the formula.
\begin{equation}
\begin{split}
\iota_{S^2}^* \Big \{ -\triangle v^\flat \Big \} & =   -\triangle_{S^2}\left( \iota_{S^2}^* v^\flat\right) -  \iota_{S^2}^* \left\{\mathcal{L}_N\mathcal{L}_N v^\flat\right\}  + ( \kappa_1-\kappa_2 )\iota_{S^2}^*\left\{ \mathcal{L}_{N} v^\flat\right\} 
+ \frac 1{\ag^2} \iota_{S^2}^* \Big \{ \mathcal{L}_{\mathrm{Y}}  v^\flat  \Big\}\\
% \sqrt{K_{S}2
&\qquad+ 2(\kappa_2-\kappa_1)(\mathcal L_{N} v^\flat)_1 E^1-2\left(\frac{E_1(\ag)}{\ag}\right)^2v^\flat_1E^1.
\end{split}
\end{equation}
Since $\kappa_i=-1$, $i=1,2$, $\ag=1$, $N=\partial_\rho$ we immediately have
\begin{equation}
\begin{split}
\iota_{S}^* \Big \{ -\triangle v^\flat \Big \} & =   -\triangle_{S}\left( \iota_{S}^* v^\flat\right) -  \iota_{S}^* \left\{\mathcal{L}_{\partial_\rho}\mathcal{L}_{\partial_\rho} v^\flat\right\}  
+  \iota_{S}^* \Big \{ \mathcal{L}_{\mathrm{Y}}  v^\flat  \Big\},
% \sqrt{K_{S}}
\end{split}
\end{equation}
 and from $Y=\ag E_1{\ag}=\ag E_1(1)=0$, we get
\begin{equation}\label{formulaS2}
\begin{split}
\iota_{S}^* \Big \{ -\triangle v^\flat \Big \} & =   -\triangle_{S}\left( \iota_{S}^* v^\flat\right) -  \iota_{S}^* \left\{\mathcal{L}_{\partial_\rho}\mathcal{L}_{\partial_\rho} v^\flat\right\} .
% \sqrt{K_{S}}
\end{split}
\end{equation}
 Similarly, the formula \eqref{formulaSR2} reduces to \eqref{formulaS2} as well.
  
\subsection{Theorem \ref{thm2}: ellipsoid}
Here we show how the formula \eqref{formulaSR} is equivalent to the formula established in \cite{CCY_22}.  The formula in \cite{CCY_22} reads as follows
\begin{equation}\label{formula}
\begin{split}
\iota_E^* \left \{-\triangle v^\flat \right\} &=   -\triangle_{\mathrm{E}}\left (\iota_E^* v^\flat  \right ) + \mathcal{\tilde E} \left ( v^\flat \right )
- \sqrt{K_{\mathrm{E}}} \iota_E^* \left\{ \mathcal{L}_{\mathrm{\tilde Y}}  v^\flat  \right \}\\
&\qquad- 2 \sqrt{K_{\mathrm{E}}} \big (1- \frac{1}{|\nabla \rho |^2} \big ) (\mathcal L_{\nabla \rho} v^\flat)_1  E^1,
\end{split}
\end{equation}
where
\begin{itemize}
\item $\mathfrak i^\ast_E$ is the pullback by the inclusion map $\mathfrak i_E: E\hookrightarrow \R^3$;
\item  $ -\triangle_{\mathrm{E}}$ is the Hodge Laplacian on $E$;
\item $\mathcal{\tilde E}$ is an operator given by
\begin{equation}
\mathcal{\tilde E} =\iota_E^* \Big \{ -\mathcal{L}_{\nabla \rho} \Big ( \frac{1}{|\nabla \rho|^2}  \mathcal{L}_{\nabla \rho}     \Big ) + \Big ( 1 - \frac{1}{|\nabla \rho|^2}     \Big )\mathcal{L}_{\nabla \rho} \Big \} + \sqrt{\mathrm{K}_{\mathrm{E}}} \iota_E^* \Big \{ \frac{1}{|\nabla \rho|^2}  \mathcal{L}_{\nabla \rho}  \Big \} ,
\end{equation}
where $K_E$ is the sectional curvature of the ellipsoid;
\item $\tilde Y=\rho \partial_{\rho}$, for $\sigma$ a defining function we show below;
\item $E^1$ is the $1$-form on $E$ dual to  $E_2$, which is a unit vector field in the direction that points along the meridians on the ellipsoid;
\end{itemize}
Since formula \eqref{formulaSR} is equivalent to formula \eqref{formulaSR2}, we compare \eqref{formula} to \eqref{formulaSR2}, and we see it is enough to show
\be\label{e1}
\begin{split}
&\Big ( 1 - \frac{1}{|\nabla \rho|^2}     \Big )\mathcal{L}_{\nabla \rho}v^\flat 
+ \sqrt{\mathrm{K}_{\mathrm{E}}} \frac{1}{|\nabla \rho|^2}  \mathcal{L}_{\nabla \rho} v^\flat
- \sqrt{K_{\mathrm{E}}}   \mathcal{L}_{\mathrm{\tilde Y}}  v^\flat   \\
 &\quad=\Big( \frac{\kappa_1-\kappa_2}{\ag}-\frac{N(\ag)}{\ag^2}\Big)\mathcal{L}_{\nabla \rho} v^\flat + \frac 1{\ag^2}\iota_E^* \Big \{ \mathcal{L}_{\mathrm{Y}}  v^\flat   \Big \},
% \sqrt{K_{\mathrm{E}}}
\end{split}
\ee
and
\be\label{e2}
-\sqrt{K_{\mathrm{E}}} \big (1- \frac{1}{|\nabla \rho |^2} \big ) (\mathcal L_{\nabla \rho} v^\flat)_1  E^1=\frac 1{\ag}(\kappa_2-\kappa_1)  (\mathcal L_{\nabla \rho} v^\flat)_1  E^1.
\ee
First, let $a>0$, then the ellipsoid
\[
E=\{ x^2+y^2+a^2z^2=a^2\},
\]
can be parametrized by a map
 
 \begin{equation}
\Phi(\rho, \phi, \theta)=(a\rho\sin \phi \cos\theta, a\rho \sin\phi \sin\theta, \rho\cos \phi),
\end{equation}
when $\rho=1$, so then $\rho$ is the (global) defining function.  Observe that $E$ is a surface of revolution obtained by revolving around the $z$-axis a curve parametrized by $(a \sin\phi, \cos \phi)$.  

 A computation shows
\[
\ag^2=\frac{\lambda^2}{a^2}:=\frac{a^2\cos^2\phi+\sin^2\phi}{a^2},
\]
and
\[
K_E=\frac 1{\lambda^4},
\]
 and the principal curvatures are given by
 \[
 \kappa_1=-\frac{a}{\lambda^3}, \quad \kappa_2=-\frac{1}{a\lambda}.
 \]
 Now it is easy to verify \eqref{e2} since
 \[
 -\sqrt{K_E}(1-\frac 1{\ag^2})=-\frac 1{\lambda^2}(1-\frac{a^2}{\lambda^2})=\frac a{\lambda}(-\frac{1}{a\lambda}+\frac{a}{\lambda^3})=\frac 1{\ag}(\kappa_2-\kappa_1) .
 \]
 The verification of \eqref{e1} uses \eqref{drho} and is left to an interested reader.

\bibliography{ref}

\def\cprime{$'$} \def\cprime{$'$}
\begin{thebibliography}{10}

\bibitem{Batchelor}
G.~K. Batchelor.
\newblock {\em An introduction to fluid dynamics}.
\newblock Cambridge Mathematical Library. Cambridge University Press,
  Cambridge, paperback edition, 1999.

\bibitem{CCM17}
Chi~Hin Chan, Magdalena Czubak, and Marcelo~M. Disconzi.
\newblock The formulation of the {N}avier-{S}tokes equations on {R}iemannian
  manifolds.
\newblock {\em J. Geom. Phys.}, 121:335--346, 2017.

\bibitem{CCY_22}
Chi~Hin Chan, Magdalena Czubak, and Tsuyoshi Yoneda.
\newblock The restriction problem on the ellipsoid.
\newblock {\em J. Math. Anal. Appl.}, 527(1):Paper No. 127358, 17, 2023.

\bibitem{Clelland_book}
Jeanne~N. Clelland.
\newblock {\em From {F}renet to {C}artan: the method of moving frames}, volume
  178 of {\em Graduate Studies in Mathematics}.
\newblock American Mathematical Society, Providence, RI, 2017.

\bibitem{C24}
Magdalena Czubak.
\newblock In search of the viscosity operator on {R}iemannian manifolds.
\newblock {\em Notices Amer. Math. Soc.}, 71(1):8--16, 2024.

\bibitem{EbinMarsden}
David~G. Ebin and Jerrold Marsden.
\newblock Groups of diffeomorphisms and the motion of an incompressible fluid.
\newblock {\em Ann. of Math. (2)}, 92:102--163, 1970.

\bibitem{GHL}
Sylvestre Gallot, Dominique Hulin, and Jacques Lafontaine.
\newblock {\em Riemannian geometry}.
\newblock Universitext. Springer-Verlag, Berlin, third edition, 2004.

\bibitem{KB2}
Shoshichi Kobayashi and Katsumi Nomizu.
\newblock {\em Foundations of differential geometry. {V}ol. {II}}.
\newblock Wiley Classics Library. John Wiley \& Sons, Inc., New York, 1996.
\newblock Reprint of the 1969 original, A Wiley-Interscience Publication.

\bibitem{Lee_RG}
John~M. Lee.
\newblock {\em Introduction to {R}iemannian manifolds}, volume 176 of {\em
  Graduate Texts in Mathematics}.
\newblock Springer, Cham, 2018.
\newblock Second edition of [ MR1468735].

\bibitem{Orszag}
Steven~A. Orszag.
\newblock Fourier series on spheres.
\newblock {\em Monthly Weather Review}, 102:56--75, 1974.

\bibitem{Taylor3}
Michael~E. Taylor.
\newblock {\em Partial differential equations {III}. {N}onlinear equations},
  volume 117 of {\em Applied Mathematical Sciences}.
\newblock Springer, New York, second edition, 2011.

\bibitem{Yamada_2018}
Michio Yamada.
\newblock On the 2-dimensional {N}avier-{S}tokes equations on 2d sphere.
\newblock {\em Bulletin of the Japan Society for Industrial and Applied
  Mathematics}, 28(1):42--45, 2018.

\end{thebibliography}
\bibliographystyle{plain}

\end{document}